\documentclass[11pt,a4paper]{article}

\usepackage[T1]{fontenc}
\usepackage[utf8]{inputenc}
\usepackage{lmodern}
\usepackage{microtype}
\usepackage[a4paper,margin=2.75cm]{geometry}
\usepackage{amsmath,amssymb,amsthm,mathtools}
\usepackage{enumitem}
\usepackage{xcolor}
\usepackage[colorlinks=true,linkcolor=blue!55!black,
  citecolor=green!40!black,urlcolor=blue!55!black]{hyperref}
\usepackage[nameinlink,noabbrev]{cleveref}
\numberwithin{equation}{section}
\allowdisplaybreaks[2]

\newcommand{\R}{\mathbb R}
\newcommand{\C}{\mathbb C}
\newcommand{\Q}{\mathbb Q}
\newcommand{\N}{\mathbb N}
\newcommand{\E}{\mathbb E}
\newcommand{\Law}{\operatorname{Law}}
\newcommand{\Sym}{\operatorname{Sym}}
\newcommand{\Contr}{\operatorname{Contr}}
\newcommand{\Id}{\operatorname{Id}}
\newcommand{\Ai}{\operatorname{Ai}}
\newcommand{\Span}{\operatorname{span}}

\newcommand{\dd}{\,\mathrm d}

\newcommand{\cP}{\mathcal P}
\newcommand{\cS}{\mathcal S}

\newcommand{\cH}{\mathcal H}

\newcommand{\ip}[2]{\left\langle #1,#2\right\rangle}

\newtheorem{theorem}{Theorem}[section]
\newtheorem{proposition}[theorem]{Proposition}
\newtheorem{lemma}[theorem]{Lemma}
\newtheorem{corollary}[theorem]{Corollary}
\theoremstyle{definition}
\newtheorem{definition}[theorem]{Definition}
\newtheorem{example}[theorem]{Example}
\theoremstyle{remark}
\newtheorem{remark}[theorem]{Remark}

\title{\vspace{-1.2cm}\textbf{Correlated Gaussian Moments Determine Polynomials up to Orthogonal Equivalence}}
\author{Obayda Julien Assaad\\
\small Independent researcher\\
\small Chisinau, Moldova\\
\small Corresponding author: \href{mailto:Obayda.assaad@gmail.com}{Obayda.assaad@gmail.com}}
\date{October 2026}
\hypersetup{
  pdftitle={Correlated Gaussian Moments Determine Polynomials up to Orthogonal Equivalence},
  pdfauthor={Obayda Julien Assaad},
  pdfsubject={Finite Gaussian orbit reconstruction and its realisation in polynomial oscillatory periods},
  pdfkeywords={Gaussian polynomials, orthogonal invariants, correlated moments, twisted de Rham cohomology, oscillatory periods, Stokes phenomena}
}

\begin{document}
\maketitle

\begin{abstract}
Let \(P\) be a real polynomial of degree at most \(m\) on \(\R^d\).
We construct finitely many correlated Gaussian moments that determine
\(P\) up to an orthogonal change of variables. The number of copies
and the rational covariance matrices are specified in terms of \(d,m\).
On coefficient balls, the induced map on orthogonal orbits has a
H\"older continuous inverse with an explicit exponent.

For \(m\geq3\), the normalised local stationary-phase expansion at the
critical branch through the origin equals the Gaussian moment series.
Its coefficients recover the same finite moment vector, without
assumptions on the other critical points. If the leading homogeneous
part has an isolated complex critical point, we compute the twisted
de Rham ranks and construct formal lattices on each fixed coupling support.
\end{abstract}

\medskip
\noindent\textbf{Keywords.}
Gaussian polynomials; orthogonal invariants; correlated Gaussian moments;
twisted de Rham cohomology; oscillatory periods; Stokes phenomena; orbital
stability.

\medskip
\noindent\textbf{2020 Mathematics Subject Classification.}
Primary 60E05, 13A50; secondary 32S40, 34M40, 60H05.

\tableofcontents

\section{Introduction}
\label{sec:introduction}

\subsection{The question and the construction}

\emph{Can a real polynomial be determined, up to an orthogonal change
of variables, by finitely many mixed moments of its values at correlated
Gaussian vectors?}

Fix \(d,m\geq1\), and let
\[
  \cP_{d,m}=\R[x_1,\ldots,x_d]_{\leq m}.
\]
We seek a finite family of such moments that distinguishes every
orthogonal orbit in \(\cP_{d,m}\).  The number of Gaussian vectors
and their covariance matrices must be chosen using only \(d,m\),
independently of the unknown polynomial.
If \(X\) is standard Gaussian in \(\R^d\), then \(P(X)\) and
\(P(UX)\) have the same law for every \(U\in O(d)\).  Orthogonal
equivalence is therefore the natural ambiguity.  The marginal law alone
does not determine this equivalence class: for
\[
  P(x,y)=x(x^2+y^2),\qquad Q(x,y)=x^3-3xy^2,
\]
the polar decomposition \(X=(R\cos\Theta,R\sin\Theta)\) gives
\[
  P(X)=R^3\cos\Theta,\qquad Q(X)=R^3\cos(3\Theta).
\]
These variables have the same law, whereas \(P\) and \(Q\) are not
orthogonally equivalent: \(\Delta P=8x\) and \(\Delta Q=0\).

We instead use correlated copies.  For \(r\geq1\) and
\(\Sigma\in\Sym_r^{++}(\R)\), let
\(X^\Sigma=(X_1,\ldots,X_r)\) be centred Gaussian with
\[
  \E[X_{a,i}X_{b,j}]=\Sigma_{ab}\delta_{ij},
\]
and define
\begin{equation}
  M_{P,r}(\Sigma)=\E\prod_{a=1}^rP(X_a).
  \label{eq:intro-replica-moment}
\end{equation}
Both diagonal and off-diagonal covariances vary, to record
self-contractions as well as contractions between copies.

These moments are derivatives at zero of Gaussian integrals with the
oscillatory factor \(e^{i\sum_a\tau_aP(x_a)}\).  For \(m\geq3\),
we will recover the same finite moment data from the local
stationary-phase expansion at the critical branch issuing from the
Gaussian origin.  This is the purpose of introducing the exponential
periods defined in \cref{subsec:necessity-architecture}: their local
expansions provide a second realisation of the orbit invariants.
Our construction has three parts.
\begin{enumerate}[label=\textup{(\arabic*)},leftmargin=*]
\item \emph{Recover the orbit from finitely many moments.}
Wick's formula expresses \(M_{P,r}\) as a polynomial in the covariance
entries, whose coefficients are explicit integer multiples of complete
contractions of the tensors of \(P\).  We recover those coefficients on explicit rational
positive-definite grids and use orthogonal invariant theory to bound
the required number of copies.  This gives a finite separating map,
together with a H\"older estimate for its inverse on bounded sets
(\cref{sec:replica-contractions,sec:orbit-recovery}).

\item \emph{Compute the differential systems for the periods.}
We reduce the twisted de Rham complexes of the Gaussian polynomial
phases to finite-dimensional Jacobian algebras.  When the leading
homogeneous part has an isolated complex critical point, we compute
the rank for each set of nonzero couplings and construct a formal
lattice on each such support.  Gaussian integration accounts for the
inactive variables
(\cref{sec:quadratic-comparison,sec:global-comparison,sec:zero-coupling}).

\item \emph{Read the moment vector from the origin period.}
For \(m\geq3\), radial rescaling singles out the critical branch
through the origin.  We prove that its normalised local formal period
is the Gaussian moment series, so finitely many coefficients recover
the same orthogonal orbit.  We then determine when this local period
is obtained by projecting a continued real-contour period onto its
origin summand (\cref{prop:marked-Gaussian-level,thm:normal-symbol-torelli}).
\end{enumerate}
The finite reconstruction holds for every \(P\in\cP_{d,m}\), as does
the local period identity when \(m\geq3\).  The isolation and Morse assumptions are stated
separately where needed for the cohomology and global period results.

\subsection{Finite orbit reconstruction and stability}

Write
\begin{equation}
  N_{d,m}=\binom{d+m}{m},
  \qquad
  r_*(d,m)=
  \left\lceil
  \max\left\{
    2,\frac38N_{d,m}\,2^{d(d+1)}m^{d(d-1)}
  \right\}
  \right\rceil .
  \label{eq:closed-replica-cutoff}
\end{equation}
For \(r\geq1\), set
\[
  D_r=\left\lfloor\frac{rm}{2}\right\rfloor,
  \qquad n_r=\frac{r(r+1)}2.
\]
Let \(E_{ab}\) denote the standard matrix units.  If \(D_r=0\), put
\(\cS_r=\{I_r\}\); otherwise put
\begin{equation}
  \cS_r=
  \left\{
  I_r+\frac1{2D_r}
  \left(
    \sum_{a=1}^ru_{aa}E_{aa}
    +\sum_{1\leq a<b\leq r}u_{ab}(E_{ab}+E_{ba})
  \right):
  u_{ab}\in\N,\ \sum_{a\leq b}u_{ab}\leq D_r
  \right\}.
  \label{eq:intro-explicit-spd-grid}
\end{equation}
Define the moment map
\begin{equation}
  \mathbf J_{d,m}(P)
  =
  \bigl(M_{P,r}(\Sigma)\bigr)_{
    1\leq r\leq r_*(d,m),\ \Sigma\in\cS_r}.
  \label{eq:intro-finite-probe-map}
\end{equation}

\begin{theorem}[Finite Gaussian orbit reconstruction]
\label{thm:intro-torelli}
Let \(d,m\geq1\) and \(P,Q\in\cP_{d,m}\). The following are equivalent:
\begin{enumerate}[label=\textup{(\roman*)}]
\item \(Q=P\circ U\) for some \(U\in O(d)\);
\item for every \(1\leq r\leq r_*(d,m)\) and \(\Sigma\in\cS_r\),
\[
  \Law\bigl(P(X_1^\Sigma),\ldots,P(X_r^\Sigma)\bigr)
  =
  \Law\bigl(Q(X_1^\Sigma),\ldots,Q(X_r^\Sigma)\bigr);
\]
\item \(\mathbf J_{d,m}(P)=\mathbf J_{d,m}(Q)\).
\end{enumerate}
\end{theorem}

Every matrix in \(\cS_r\) is rational and satisfies
\begin{equation}
  \frac12I_r\preceq\Sigma\preceq\frac32I_r,
  \qquad
  |\cS_r|=\binom{D_r+n_r}{n_r}.
  \label{eq:intro-grid-size-conditioning}
\end{equation}
All complete contractions on at most \(r_*(d,m)\) tensor copies are
recovered from \(\mathbf J_{d,m}(P)\) by rational finite differences.  In particular the number of
scalar observations is
\begin{equation}
  N_{\mathrm{probe}}(d,m)
  =
  \sum_{r=1}^{r_*(d,m)}
  \binom{\lfloor rm/2\rfloor+r(r+1)/2}{r(r+1)/2}.
  \label{eq:probe-count}
\end{equation}

The bound \(r_*(d,m)\) is sufficient and need not be optimal.
The proof is in \cref{sec:replica-contractions,sec:orbit-recovery}.
On each coefficient ball, the moment map is Lipschitz and its inverse on
orthogonal orbits is H\"older continuous; \cref{thm:quantitative-torelli}
gives an explicit exponent.

\subsection{Exponential periods and oscillatory integrals}
\label{subsec:necessity-architecture}

Put \(N=rd\) and \(\boldsymbol x=(x_1,\ldots,x_r)\in\C^N\).
The polynomial phase studied here is
\[
 F_{\Sigma,\boldsymbol\tau}(\boldsymbol x)
 =-\frac12\boldsymbol x^{\mathsf T}
       (\Sigma^{-1}\otimes I_d)\boldsymbol x
   +i\sum_{a=1}^r\tau_aP(x_a).
\]
An \emph{exponential period} of this phase is an integral
\[
 \Pi_{\Gamma,f}(\boldsymbol\tau)
 =\int_\Gamma f(\boldsymbol x)
   e^{F_{\Sigma,\boldsymbol\tau}(\boldsymbol x)}
   \,\operatorname{vol}_{\boldsymbol x}.
\]
Here \(f\) is a polynomial amplitude,
\(\operatorname{vol}_{\boldsymbol x}
=\bigwedge_{a=1}^r\bigwedge_{j=1}^d\mathrm dx_{a,j}\), and
\(\Gamma\) is an oriented real \(N\)-dimensional integration cycle
in \(\C^N\) with rapid-decay ends: the exponential decays faster
than every inverse power of \(\|\boldsymbol x\|\) along those ends.
The word \emph{period} refers to the pairing of twisted de Rham
cohomology with rapid-decay homology \cite{HienRoucairol}.
Indeed, writing \(F=F_{\Sigma,\boldsymbol\tau}\), Stokes' formula gives
\[
 \int_\Gamma e^F(\mathrm d+\mathrm dF\wedge)\eta
 =\int_\Gamma\mathrm d(e^F\eta)=0
\]
for every polynomial \((N-1)\)-form \(\eta\); rapid decay removes
the boundary term.  The integral therefore depends only on the
twisted cohomology class of \(f\operatorname{vol}_{\boldsymbol x}\)
and the homology class of \(\Gamma\).

For real \(\boldsymbol\tau\), the cycle \(\Gamma=\R^N\) is
admissible: the factor \(e^{i\sum_a\tau_aP(x_a)}\) is oscillatory,
and the Gaussian factor ensures convergence.  With amplitude \(f=1\)
and \(Z_\Sigma=(2\pi)^{-N/2}(\det\Sigma)^{-d/2}\), its normalised
period is
\begin{align}
 \Phi_{P,r,\Sigma}(\boldsymbol\tau)
 &=Z_\Sigma\int_{\R^N}
 e^{-\boldsymbol x^{\mathsf T}(\Sigma^{-1}\otimes I_d)\boldsymbol x/2}
 e^{i\sum_a\tau_aP(x_a)}\,\operatorname{vol}_{\boldsymbol x}
 \notag\\
 &=\E\exp\!\left(i\sum_{a=1}^r\tau_aP(X_a^\Sigma)\right).
 \label{eq:intro-characteristic-period}
\end{align}
The other periods considered below use descending cycles attached to
critical points of the phase, in place of the real cycle.
Differentiation under the Gaussian integral gives
\begin{equation}
 i^{-r}\partial_{\tau_1}\cdots\partial_{\tau_r}
 \Phi_{P,r,\Sigma}(0)=M_{P,r}(\Sigma).
 \label{eq:intro-period-jet}
\end{equation}
The question is therefore to identify these derivatives within the
stationary-phase expansions of the period.  The critical branch through
the origin suffices: its local formal expansion contains every Gaussian
moment, even when other critical points are degenerate or share its
leading critical value.

Write \(P=\sum_{k=0}^mP_k\), with \(P_k\) homogeneous, and assume
\(m\geq3\).  Set \(\tau_a=\rho^{m-2}\lambda_a\) and
\(x_a=\rho^{-1}u_a\).  The phase becomes
\(\rho^{-2}\Psi(\rho,\boldsymbol u)\), where
\begin{equation}
 \Psi(\rho,\boldsymbol u)
 =-\frac12\langle\boldsymbol u,
          (\Sigma^{-1}\otimes I_d)\boldsymbol u\rangle
   +i\sum_{a=1}^r\lambda_a\sum_{k=0}^m\rho^{m-k}P_k(u_a).
 \label{eq:intro-rescaled-phase}
\end{equation}
Its leading term is
\[
 G(\boldsymbol u)=\Psi(0,\boldsymbol u)
 =-\frac12\langle\boldsymbol u,
          (\Sigma^{-1}\otimes I_d)\boldsymbol u\rangle
   +i\sum_{a=1}^r\lambda_aP_m(u_a).
\]
The origin is a nondegenerate critical point of \(G\).  Let
\(u_0(\rho)\) be its formal continuation as a critical point of \(\Psi\).
Denote by \(\widehat\Pi_0\) the formal stationary-phase expansion at
\(u_0(\rho)\) of
\[
 (2\pi)^{-rd/2}(\det\Sigma)^{-d/2}\rho^{-rd}
 \int_{\Gamma_0(\rho)}
 e^{\rho^{-2}\Psi(\rho,\boldsymbol u)}
 \,\operatorname{vol}_{\boldsymbol u}.
\]
Here \(\Gamma_0\) is the local descending cycle, oriented by the real
Gaussian cycle at \(\rho=0\); the normalisation gives
\(\widehat\Pi_0=1+O(\rho)\).

\begin{theorem}[The formal period at the origin]
\label{thm:intro-period-realisation}
Let \(P\in\cP_{d,m}\), \(m\geq3\), \(r\geq1\), and
\(\Sigma\in\Sym_r^{++}(\R)\).  For every
\(\boldsymbol\lambda\in\C^r\), the local critical branch with \(u_0(0)=0\)
is nondegenerate. Its normalised formal period satisfies
\begin{equation}
 \widehat\Pi_0(\rho,\boldsymbol\lambda)
 =\sum_{\alpha\in\N^r}
   \frac{i^{|\alpha|}\boldsymbol\lambda^\alpha}{\alpha!}
   \rho^{(m-2)|\alpha|}
   \E\prod_{a=1}^rP(X_a^\Sigma)^{\alpha_a}.
 \label{eq:intro-period-moment-series}
\end{equation}
If, in addition,
\begin{equation}
 \{x\in\C^d:\partial_1P_m(x)=\cdots=\partial_dP_m(x)=0\}=\{0\},
 \label{eq:intro-principal-isolation}
\end{equation}
\(\boldsymbol\lambda\in(\C^*)^r\), \(G\) is Morse, and
\(\operatorname{Crit}(G)\cap G^{-1}(0)=\{0\}\), then the formal radial
Gauss--Manin summand indexed by critical value zero has rank one, and
\(\widehat\Pi_0\) is its normalised period.
\end{theorem}

Thus \(\widehat\Pi_0\in\C[\boldsymbol\lambda][[\rho]]\), and
\begin{equation}
 M_{P,r}(\Sigma)
 =i^{-r}[\rho^{(m-2)r}\lambda_1\cdots\lambda_r]\,
   \widehat\Pi_0.
 \label{eq:intro-normal-symbol-identity}
\end{equation}
The local formula is proved by the Gaussian expansion and integration by
parts in \cref{prop:marked-Gaussian-level}; it uses only the invertible
Hessian \(-\Sigma^{-1}\otimes I_d\) at the origin.
The assertion about the intrinsic summand follows from
\cref{prop:radial-formal-blocks}.  We denote the series in
\eqref{eq:intro-period-moment-series} by
\(\widehat\Phi^{\mathrm G}_{P,r,\Sigma}\).

For the full period system, \cref{thm:gaussian-face-cube,thm:toric-corner-Rees}
give the ranks and formal lattices on each coupling support, and
\cref{thm:thimble-factorisation} separates the choice of contour from
the thimble periods.  Under the rank-one and contour-coefficient
hypotheses of \cref{thm:normal-symbol-torelli}, projection onto the
origin summand recovers \(\widehat\Phi^{\mathrm G}_{P,r,\Sigma}\).
The local coefficient formula \eqref{eq:intro-normal-symbol-identity}
does not require these hypotheses.

\subsection{Previous work}

Wick's formula and Gaussian diagram expansions are classical
\cite{Isserlis,Janson,PeccatiTaqqu}.  Complete contractions generate
orthogonal invariants by the first fundamental theorem
\cite{Weyl}.  For one real symmetric tensor of each order, Schrijver's graph
formulation identifies equality of all contractions with orthogonal
equivalence \cite[Theorem~2]{SchrijverEdge}.  Young gives the corresponding
criterion for real tensor families without symmetry assumptions
\cite[Theorem~9]{YoungHolant}.  We use these separation
results together with Derksen's invariant-degree bound
\cite[Theorem~1.1 and Proposition~1.2]{DerksenBounds} to prescribe finitely
many Gaussian moments.  The covariance grids and finite-difference
formulas are given in \cref{lem:spd-grid,cor:finite-contraction-recovery}.
The stability estimate applies the polynomial \L ojasiewicz inequality of
\cite[Corollary~10]{KurdykaSpodzieja} to this moment map.

Polynomial orbit spaces and constructive invariant theory are treated in
\cite{Schwarz,ProcesiSchwarz,DerksenKemper}.  The principal-component
method of Helmer, Hong and Hong \cite{HelmerHongHong} reconstructs an
orthogonal equivalence on a generic locus.  The finite moment map used
here separates every orbit in \(\cP_{d,m}\); the
size estimate does not address computational efficiency.

The period calculations use the theory of isolated singularities and
Gauss--Manin systems \cite{Milnor,Brieskorn,DimcaSaito,DouaiSabbah,SabbahTwisted},
homological perturbation \cite{GugenheimLambe,GugenheimLambeStasheff},
and the pairing between twisted de Rham cohomology and rapid-decay homology
\cite{Pham,Hien,HienRoucairol}.  We apply these results to the correlated
Gaussian phase, compute its changes of rank when couplings vanish, and
identify its formal Gaussian period with the moment series.  Stokes
structures and related Gaussian examples are discussed in
\cite{SabbahStokes,DAgnoloHienMorandoSabbah,Hohl}.

The appendices provide explicit examples and supplementary constructions.
Appendix~\ref{sec:cubic} computes the cubic Airy periods and their full
period determinant. Appendix~\ref{sec:quartic} gives the quartic Bessel
normalisation and the parity obstruction to cyclicity of the real contour.
Appendix~\ref{sec:one-variable-refinements} computes monomial observable
ranks and residue pairings, and specialises the formal lattice to one copy.
Appendix~\ref{sec:one-dimensional-stokes-trees} supplies the one-variable
cycle-coordinate calculations used in \cref{sec:real-contour-periods};
Appendix~\ref{sec:selected-contour-cyclicity} states the selected-contour
cyclicity criterion. The finite orbit-separation theorem is proved entirely in
\cref{sec:replica-contractions,sec:orbit-recovery}.


\section{Correlated Gaussian moments and complete contractions}
\label{sec:replica-contractions}

Let \(V=\R^d\) with its Euclidean inner product. Every
\(P\in\cP_{d,m}\) has a unique tensor expansion
\begin{equation}
 P(x)=\sum_{k=0}^{m}\ip{T_k(P)}{x^{\otimes k}},
 \qquad T_k(P)\in\Sym^k(V).
 \label{eq:tensor-expansion}
\end{equation}
In an orthonormal basis, this means
\[
 P(x)=T_0(P)+\sum_{k=1}^{m}\sum_{i_1,\ldots,i_k=1}^{d}
 (T_k(P))_{i_1\cdots i_k}x_{i_1}\cdots x_{i_k},
 \qquad
 (T_k(P))_{i_1\cdots i_k}
 =\frac1{k!}\partial_{i_1}\cdots\partial_{i_k}P(0).
\]
Thus \(T_0(P)=P(0)\), \(T_1(P)\) is the vector of linear
coefficients, and \(T_2(P)=A\) when the quadratic part is
\(x^{\mathsf T}Ax\), with \(A\) symmetric.

We use a graph to specify which tensor indices are summed together.
A \emph{looped multigraph} on the labelled vertices
\([r]=\{1,\ldots,r\}\) is specified by nonnegative integers
\[
 E=(e_{ab})_{1\leq a\leq b\leq r}.
\]
For \(a<b\), \(e_{ab}\) counts the edges joining \(a\) to \(b\);
\(e_{aa}\) counts the loops at \(a\). Each ordinary edge has one
end at each of its two vertices, whereas a loop has both ends at the
same vertex. The number of edge ends at \(a\) is therefore
\begin{equation}
 k_a(E)=2e_{aa}+\sum_{b\neq a}e_{\min(a,b),\max(a,b)}.
 \label{eq:graph-degree}
\end{equation}
We call \(E\) admissible if \(k_a(E)\leq m\) for every \(a\).

Here is an explicit definition of the \emph{complete contraction}
\(\Contr_E(P)\). At vertex \(a\), use the tensor
\(T_{k_a(E)}(P)\), which has one index for each edge end at that vertex.
Assign an index in \(\{1,\ldots,d\}\) to each edge, use that same
index at its two ends, multiply the resulting tensor entries over all
vertices, and sum over all edge indices. In particular, a loop inserts
the same index twice into one tensor.

More formally, temporarily distinguish the individual edges, including
parallel edges, and denote their set by \(\mathcal E(E)\). At each
vertex \(a\), list the incident edges as
\(\epsilon_{a,1},\ldots,\epsilon_{a,k_a(E)}\), listing each loop
twice. Then
\begin{equation}
 \Contr_E(P)
 =\sum_{j:\mathcal E(E)\to\{1,\ldots,d\}}
   \prod_{a=1}^{r}
   (T_{k_a(E)}(P))_{j(\epsilon_{a,1})\cdots
                         j(\epsilon_{a,k_a(E)})}.
 \label{eq:contraction-coordinate-definition}
\end{equation}
An isolated vertex contributes the scalar \(T_0(P)\); if there are
no edges, the formula gives \(\Contr_E(P)=T_0(P)^r\).
Symmetry of each tensor makes the result independent of the order
chosen for its indices, and relabelling the edges only renames the
summation indices. This is contraction with the Euclidean metric,
whose entries in an orthonormal basis are \(\delta_{ij}\).
Every tensor index is summed, so the result is a scalar.

For example, one loop at a single vertex gives
\(\sum_i(T_2(P))_{ii}=\operatorname{tr}T_2(P)\).
A single edge between two vertices gives
\(\sum_i(T_1(P))_i^2=\|T_1(P)\|^2\), and two parallel edges
between two vertices give \(\sum_{i,j}(T_2(P))_{ij}^2\).

\begin{lemma}[Gaussian moments as a polynomial of tensor contractions]
\label{lem:replicated-wick}
Let \(P\in\cP_{d,m}\), \(r\geq1\), and
\(\Sigma\in\Sym_r^{++}(\R)\). Let
\(X^\Sigma=(X_1^\Sigma,\ldots,X_r^\Sigma)\) be the centred Gaussian
vector with \(X_a^\Sigma\in\R^d\) and
\[
 \E[X_{a,i}^\Sigma X_{b,j}^\Sigma]=\Sigma_{ab}\delta_{ij}.
\]
Then the mixed moment
\[
 M_{P,r}(\Sigma)=\E\prod_{a=1}^{r}P(X_a^\Sigma)
\]
is the restriction to positive-definite matrices of the polynomial
\begin{equation}
 M_{P,r}(\Sigma)
 =\sum_{\substack{E=(e_{ab})\\k_a(E)\leq m\ \text{for all }a}}
 N(E)\,\Contr_E(P)
 \prod_{1\leq a\leq b\leq r}\Sigma_{ab}^{e_{ab}},
 \label{eq:replicated-wick-polynomial}
\end{equation}
where the sum is over admissible looped multigraphs on \([r]\), and
\begin{equation}
 N(E)
 =\frac{\prod_{a=1}^{r}k_a(E)!}
 {\displaystyle
  \prod_{a=1}^{r}2^{e_{aa}}e_{aa}!\,
  \prod_{1\leq a<b\leq r}e_{ab}!}
 \in\N_{>0}.
 \label{eq:wick-multiplicity}
\end{equation}
The integer \(N(E)\) counts the pairings of the labelled tensor
indices that produce the graph \(E\). In particular, the coefficient
of \(\prod_{a\leq b}\Sigma_{ab}^{e_{ab}}\) is exactly
\(N(E)\Contr_E(P)\). The total degree of \(M_{P,r}\) is at most
\(\lfloor rm/2\rfloor\).

Every complete contraction of \(r\) factors chosen among the tensors
\(T_k(P)\), \(0\leq k\leq m\), is recovered in this way.
\end{lemma}

\begin{proof}
First expand each factor \(P(X_a^\Sigma)\) using
\eqref{eq:tensor-expansion}. For a fixed choice of tensor orders
\(\boldsymbol k=(k_1,\ldots,k_r)\), the contribution is
\[
 \sum_{\substack{1\leq i_{a,p}\leq d\\
                  1\leq a\leq r,\ 1\leq p\leq k_a}}
 \left(\prod_{a=1}^{r}
    (T_{k_a}(P))_{i_{a,1}\cdots i_{a,k_a}}\right)
 \E\!\left[\prod_{a=1}^{r}\prod_{p=1}^{k_a}
                      X_{a,i_{a,p}}^\Sigma\right].
\]
We regard the \(k_a\) index positions at vertex \(a\) as labelled
by \((a,1),\ldots,(a,k_a)\). Write
\(\mathcal I_{\boldsymbol k}=\{(a,p):1\leq a\leq r,\ 1\leq p\leq k_a\}\),
and let \(\mathcal P_2(\mathcal I_{\boldsymbol k})\) denote its
partitions into unordered pairs. Wick's formula, or Isserlis' formula
\cite{Isserlis}, states that
\[
 \E\!\left[\prod_{a,p}X_{a,i_{a,p}}^\Sigma\right]
 =\sum_{\pi\in\mathcal P_2(\mathcal I_{\boldsymbol k})}
   \prod_{\{(a,p),(b,q)\}\in\pi}
          \Sigma_{ab}\delta_{i_{a,p},i_{b,q}}.
\]
If \(\sum_a k_a\) is odd, there are no pairings and the expectation
is zero. For an empty set of positions there is one empty pairing,
whose contribution is one.

Each pairing defines a graph: a pair of positions at vertices \(a,b\)
gives an edge between those vertices, or a loop if \(a=b\).
The resulting graph has \(k_a(E)=k_a\). Every such pair contributes
one covariance factor \(\Sigma_{ab}\). Its Kronecker delta identifies
the two coordinate indices; summing those indices gives exactly the
contraction \eqref{eq:contraction-coordinate-definition}. Thus, after
summing over coordinates, every pairing with graph \(E\) contributes
\[
 \Contr_E(P)\prod_{a\leq b}\Sigma_{ab}^{e_{ab}}.
\]
Symmetry of the tensors ensures that this contribution depends only
on the graph, not on which labelled positions form its edges.
It remains to count how many pairings give the same graph.

Fix \(E\), and abbreviate \(e_{ba}=e_{ab}\) when \(b>a\).
The count has three steps:
\begin{enumerate}[label=\textup{(\roman*)},leftmargin=*]
\item At each vertex \(a\), choose the \(2e_{aa}\) positions used
for loops and, for each \(b\neq a\), the \(e_{ab}\) positions to be
paired with positions at \(b\). The number of choices is
\[
 \prod_{a=1}^{r}
 \frac{k_a(E)!}{(2e_{aa})!\prod_{b\neq a}e_{ab}!}.
\]
\item Pair the \(2e_{aa}\) loop positions at each vertex. This gives
\[
 \prod_{a=1}^{r}\frac{(2e_{aa})!}{2^{e_{aa}}e_{aa}!}
\]
choices: divide by two for the order within each pair, and by
\(e_{aa}!\) for the order of the pairs.
\item For each \(a<b\), match the \(e_{ab}\) chosen positions at
\(a\) bijectively with the \(e_{ab}\) chosen positions at \(b\).
There are \(\prod_{a<b}e_{ab}!\) choices.
\end{enumerate}
These choices determine a pairing uniquely. Multiplying the three
counts cancels the factors \((2e_{aa})!\). Each \(e_{ab}!\) with
\(a<b\) occurs twice in the denominator of the first count and once
in the numerator of the third, leaving precisely
\eqref{eq:wick-multiplicity}. Grouping all pairings by their graphs
proves \eqref{eq:replicated-wick-polynomial}.

Different arrays \((e_{ab})\) give different covariance monomials,
which proves the coefficient statement. Moreover, every edge uses
two positions, so
\[
 \sum_{a\leq b}e_{ab}=\frac12\sum_a k_a(E)\leq\frac{rm}{2}.
\]
This proves the degree bound. Finally, any complete pairing of tensor
indices defines such a graph, including isolated vertices for scalar
factors \(T_0(P)\), which proves the last assertion.
\end{proof}

For a concrete instance, take \(r=2\) and
\(P(x)=x^{\mathsf T}Ax\), with \(A\) symmetric. Only graphs of
degree two at both vertices contribute. There are two possibilities:
a loop at each vertex, or two edges joining the vertices. The first
has contraction \((\operatorname{tr}A)^2\) and multiplicity one;
the second has contraction \(\sum_{i,j}A_{ij}^2=\operatorname{tr}(A^2)\)
and multiplicity two. Hence
\[
 \E\bigl[(X_1^{\mathsf T}AX_1)(X_2^{\mathsf T}AX_2)\bigr]
 =\Sigma_{11}\Sigma_{22}(\operatorname{tr}A)^2
   +2\Sigma_{12}^{2}\operatorname{tr}(A^2).
\]
The powers of the covariance entries record the edges, while their
coefficients recover the two tensor contractions.

\begin{lemma}[Interpolation on positive-definite matrices]
\label{lem:spd-grid}
Put
\[
 D_r=\left\lfloor\frac{rm}{2}\right\rfloor,
 \qquad n_r=\frac{r(r+1)}2.
\]
The set \(\cS_r\) defined in \eqref{eq:intro-explicit-spd-grid} is contained in
\(\Sym_r^{++}(\Q)\). Every \(\Sigma\in\cS_r\) satisfies
\[
 \frac12I_r\preceq\Sigma\preceq\frac32I_r,
\]
and
\[
 |\cS_r|=\binom{D_r+n_r}{n_r}.
\]
Evaluation on \(\cS_r\) is an isomorphism from the vector space of
polynomials on \(\Sym_r(\R)\) of total degree at most \(D_r\) onto
\(\R^{\cS_r}\).  Its inverse is an explicit rational finite-difference
transform.
\end{lemma}

\begin{proof}
For \(1\leq a\leq b\leq r\), put
\[
 B_{aa}=E_{aa},
 \qquad
 B_{ab}=E_{ab}+E_{ba}\quad(a<b).
\]
These \(n_r\) matrices form a rational basis of \(\Sym_r(\R)\), and each
has operator norm one.  The assertion is immediate when \(D_r=0\).  Assume
\(D_r\geq1\), enumerate the basis as \(B_1,\ldots,B_{n_r}\), and write
\[
 \Lambda_{n_r,D_r}
 =\left\{u\in\N^{n_r}:|u|\leq D_r\right\}.
\]
Then \eqref{eq:intro-explicit-spd-grid} reads
\begin{equation}
 \cS_r=
 \left\{
 I_r+\frac1{2D_r}\sum_{j=1}^{n_r}u_jB_j:
 u\in\Lambda_{n_r,D_r}
 \right\}.
 \label{eq:explicit-spd-grid}
\end{equation}
For such a point,
\[
 \left\|\frac1{2D_r}\sum_ju_jB_j\right\|_{\mathrm{op}}
 \leq\frac{|u|}{2D_r}\leq\frac12,
\]
which proves the spectral bounds.  Counting the nonnegative integer vectors
of total degree at most \(D_r\) gives
\(
 |\Lambda_{n_r,D_r}|=\binom{D_r+n_r}{n_r}
\).

Let \(F\) have total degree at most \(D_r\), and set
\[
 G_F(u)=F\!\left(I_r+\frac1{2D_r}\sum_ju_jB_j\right).
\]
Write
\[
 \binom{u}{\alpha}=\prod_{j=1}^{n_r}\binom{u_j}{\alpha_j}.
\]
In the binomial basis, the evaluation matrix \(\bigl(\binom{u}{\alpha}\bigr)_{u,\alpha\in\Lambda_{n_r,D_r}}\),
ordered by total degree, is unitriangular: its entries vanish unless
\(\alpha\leq u\), and its diagonal entries equal one.  Thus evaluation
is an isomorphism.  With \(\Delta_j\) the forward difference in coordinate
\(j\), its inverse is
\begin{align}
 (\Delta^\alpha G_F)(0)
 &=\sum_{\beta\leq\alpha}
   (-1)^{|\alpha|-|\beta|}
   \binom{\alpha}{\beta}G_F(\beta),
 \label{eq:finite-difference-inverse}\\
 G_F(u)
 &=\sum_{|\alpha|\leq D_r}
   (\Delta^\alpha G_F)(0)\binom{u}{\alpha}.
 \label{eq:newton-interpolation}
\end{align}
Both this formula and the affine change from \(G_F\) to \(F\) have
rational coefficients.
\end{proof}

\begin{corollary}[Recovery of contractions]
\label{cor:finite-contraction-recovery}
Let \(E\) be a looped multigraph on \([r]\) with
\(k_a(E)\leq m\) for every \(a\). There are explicit rational numbers
\(a_{E,\Sigma}\), independent of \(P\), such that
\begin{equation}
 \Contr_E(P)=\sum_{\Sigma\in\cS_r}a_{E,\Sigma}M_{P,r}(\Sigma)
 \qquad(P\in\cP_{d,m}).
 \label{eq:contraction-extraction}
\end{equation}
Thus the values \(M_{P,s}(\Sigma)\), with \(1\leq s\leq r\)
and \(\Sigma\in\cS_s\), determine every complete contraction
of at most \(r\) tensor copies.
\end{corollary}

\begin{proof}
Writing \(\Sigma^E=\prod_{a\leq b}\Sigma_{ab}^{e_{ab}}\),
\cref{lem:replicated-wick} gives \(\deg M_{P,r}\leq D_r\) and
\[
 [\Sigma^E]M_{P,r}=N(E)\Contr_E(P).
\]
Recover \(M_{P,r}\) from its values on \(\cS_r\) by
\cref{lem:spd-grid}. The coefficients \(a_{E,\Sigma}\) are obtained from
\eqref{eq:finite-difference-inverse}--\eqref{eq:newton-interpolation}
by extracting the coefficient of \(\prod_{a\leq b}\Sigma_{ab}^{e_{ab}}\)
and dividing by \(N(E)\).
\end{proof}

\section{Orbit separation and stability}
\label{sec:orbit-recovery}

Let \(F\) be a homogeneous polynomial of degree \(r\) on a vector space.
Its polar form is
\begin{equation}
 \widetilde F(v_1,\ldots,v_r)
 =\left.\frac1{r!}
 \partial_{s_1}\cdots\partial_{s_r}
 F\!\left(\sum_{a=1}^rs_av_a\right)\right|_{\boldsymbol s=0},
 \qquad
 F(v)=\widetilde F(v,\ldots,v).
 \label{eq:invariant-polarisation}
\end{equation}
\begin{lemma}[Orthogonal invariants]
\label{lem:orthogonal-invariants}
Let
\[
 \mathbb V_{d,m}=\bigoplus_{k=0}^{m}\Sym^k(V).
\]
The algebra \(\R[\mathbb V_{d,m}]^{O(d)}\) is generated by a finite
family of complete contractions of tensor copies. Invariant polynomials
separate the \(O(d)\)-orbits in \(\mathbb V_{d,m}\).
\end{lemma}

\begin{proof}
Apply the orthogonal first fundamental theorem \cite{Weyl} to the
polarisation \eqref{eq:invariant-polarisation}, and restrict the resulting
pairings to the symmetric tensor factors.  These restrictions are the
contractions \(\Contr_E\).  Hilbert's finite-generation theorem applies
to \(O_d(\C)\), which is reductive in characteristic zero
\cite{DerksenKemper}.  Since \(O(d)\) is Zariski dense in \(O_d(\C)\),
finite generation descends to \(\R\).
The orbit-separation statement is the graph-contraction criterion of
Schrijver \cite[Section~7, Theorem~2]{SchrijverEdge}: the real symmetric
tensors \((T_k(P))_{k\leq m}\) define an edge model, and its value on
the multigraph \(E\), including loops, is \(\Contr_E(P)\).
\end{proof}

\begin{proposition}[A bound on the number of Gaussian copies]
\label{prop:closed-replica-bound}
Let
\(
 \mathbb V_{d,m}=\bigoplus_{k=0}^{m}\Sym^k(\R^d)
\).
Then \(\R[\mathbb V_{d,m}]^{O(d)}\) is generated by homogeneous invariants
of coefficient degree at most the integer \(r_*(d,m)\) in
\eqref{eq:closed-replica-cutoff}.
\end{proposition}

\begin{proof}
For a complex representation \(W\), write \(\beta(W)\) for the least
integer such that \(\C[W]^{O_d(\C)}\) is generated by homogeneous
invariants of degree at most \(\beta(W)\).  Let \(\sigma(W)\) be the least
integer such that the invariants of positive degree at most
\(\sigma(W)\) cut out the null cone.
Embed \(O_d(\C)\) in \(\mathbb A^{d^2}\) by
\(Z^{\mathsf T}Z=I_d\).  The defining equations have degree at most
\(H=2\), the ambient affine dimension is \(t=d^2\), and
\[
 g:=\dim O_d(\C)=\frac{d(d-1)}2.
\]
On
\[
 W_{d,m}=\bigoplus_{k=0}^{m}\Sym^k(\C^d)
\]
the representation is faithful because it contains \(\Sym^1(\C^d)\), and
its matrix coefficients have degree at most \(A=m\) in the entries of
\(Z\).  Derksen's null-cone estimate \cite[Proposition~1.2]{DerksenBounds}
therefore gives
\begin{equation}
 \sigma(W_{d,m})
 \leq H^{t-g}A^g
 =2^{d(d+1)/2}m^{d(d-1)/2}.
 \label{eq:derksen-nullcone-bound}
\end{equation}
Derksen's generator estimate \cite[Theorem~1.1]{DerksenBounds} is
\[
 \beta(W)
 \leq
 \max\left\{2,\frac38
   \dim_{\mathrm{Krull}}\bigl(\C[W]^{O_d(\C)}\bigr)\sigma(W)^2
 \right\}.
\]
Since
\[
 \dim_{\mathrm{Krull}}\bigl(\C[W_{d,m}]^{O_d(\C)}\bigr)
 \leq\dim W_{d,m}
 =\sum_{k=0}^{m}\binom{d+k-1}{k}
 =\binom{d+m}{m}=N_{d,m},
\]
substitution of \eqref{eq:derksen-nullcone-bound} gives precisely
\eqref{eq:closed-replica-cutoff}.

Finally, \(O(d)\) is Zariski dense in \(O_d(\C)\), so the real invariant
algebra complexifies to the displayed complex invariant algebra.  The
generating-degree bound therefore descends to \(\R\).
\end{proof}

\begin{remark}[Minimal generating degree]
\label{rem:cutoff-hierarchy}
Put
\[
 \begin{gathered}
 A=\Q\!\left[\bigoplus_{k=0}^{m}\Sym^k(\Q^d)\right]^{O_d},
 \qquad A_+=\bigoplus_{s\geq1}A_s,\\
 r_{\mathrm{gen}}(d,m)
 =\max\left\{s\geq1:
 \left(A_+/(A_+)^2\right)_s\neq0
 \right\}.
 \end{gathered}
\]
Then \(r_{\mathrm{gen}}(d,m)\) is the largest degree in a minimal
homogeneous generating set, and
\(r_{\mathrm{gen}}(d,m)\leq r_*(d,m)\).
An algorithm for invariant generators \cite{DerksenKemper}, followed by
rational linear algebra in each degree, computes \(A_s\),
\(((A_+)^2)_s\), and hence \(r_{\mathrm{gen}}(d,m)\).
A separating family may have smaller maximal degree than a generating
family.
\end{remark}

\begin{proof}[Proof of \cref{thm:intro-torelli}]
If \(Q=P\circ U\), the simultaneous orthogonal transformation
\(X_a\mapsto UX_a\) preserves the
Gaussian law with covariance \(\Sigma\otimes I_d\), proving
\textup{(i)}\(\Rightarrow\)\textup{(ii)}.  Taking mixed moments gives
\textup{(ii)}\(\Rightarrow\)\textup{(iii)}.

Under \textup{(iii)}, \cref{cor:finite-contraction-recovery} gives
\(\Contr_E(P)=\Contr_E(Q)\) for every graph with at most
\(r_*(d,m)\) vertices.  By the orthogonal first fundamental theorem,
these contractions span all invariants of degree at most \(r_*(d,m)\).
They generate the invariant algebra by \cref{prop:closed-replica-bound}
and separate real orbits by \cref{lem:orthogonal-invariants}.  Hence
\(Q=P\circ U\) for some \(U\in O(d)\).
\end{proof}

\begin{remark}[The smallest space through which a polynomial factors]
For polynomials presented on different Euclidean spaces, define
\[
 N(P)=\bigcap_{k=1}^m\{v:\iota_vT_k(P)=0\},
 \qquad E(P)=N(P)^\perp.
\]
Then
\[
 T_k(P)\in\Sym^k(E(P)),\qquad
 P=P\circ\operatorname{proj}_{E(P)}.
\]
If \(P=P\circ\operatorname{proj}_E\), then
\(E^\perp\subseteq N(P)\), so \(E(P)\subseteq E\).
Thus \(E(P)\) is the smallest subspace through which \(P\) factors.
Moreover,
\[
 Q=P\circ U\quad\Longrightarrow\quad U E(Q)=E(P).
\]
Adjoining unused coordinates places both polynomials on the orthogonal
direct sum of their ambient spaces.  Equality of the enlarged moment maps then gives an isometry
between \(E(Q)\) and \(E(P)\), by \cref{thm:intro-torelli}.
\end{remark}

For each \(r\leq r_*(d,m)\), let \(c_r(P)\) be the monomial coefficient
vector of \(M_{P,r}\), including the Wick factors \(N(E)\), and let
\(A_r\) be evaluation on \(\cS_r\).  The evaluation matrix \(A_r\) has rational entries and is
injective by \cref{lem:spd-grid}; hence \(K_r=A_r^{\mathsf T}A_r>0\) and
\begin{equation}
 \sum_{\Sigma\in\cS_r}
 |M_{P,r}(\Sigma)-M_{Q,r}(\Sigma)|^2
 =
 \bigl(c_r(P)-c_r(Q)\bigr)^{\mathsf T}
 K_r\bigl(c_r(P)-c_r(Q)\bigr).
 \label{eq:finite-probe-gram}
\end{equation}
Set
\begin{equation}
 \mathfrak E_{\mathrm{probe}}(P,Q)
 =\|\mathbf J_{d,m}(P)-\mathbf J_{d,m}(Q)\|^2
 =\sum_{r=1}^{r_*(d,m)}\sum_{\Sigma\in\cS_r}
 |M_{P,r}(\Sigma)-M_{Q,r}(\Sigma)|^2.
 \label{eq:finite-probe-energy}
\end{equation}
By \cref{thm:intro-torelli}, this sum vanishes exactly when
\(P\) and \(Q\) are orthogonally equivalent.

Equip \(V_{d,m}=\cP_{d,m}\) with the invariant norm
\begin{equation}
 \|P\|_V^2=\sum_{k=0}^m\|T_k(P)\|_{\mathrm{HS}}^2,
 \label{eq:invariant-tensor-norm}
\end{equation}
and put
\begin{align}
 \mathbf J_{d,m}(P)
 &=\bigl(M_{P,r}(\Sigma)\bigr)_{
  1\leq r\leq r_*(d,m),\ \Sigma\in\cS_r},
 \label{eq:finite-probe-map}\\
 \delta_{\mathrm{orb}}(P,Q)
 &=\min_{U\in O(d)}\|Q-P\circ U\|_V,
 \label{eq:orbital-distance}\\
 N&=\binom{d+m}{m},\qquad D=r_*(d,m),\qquad
 A_{d,m}=D(6D-3)^{2N-1}.
 \label{eq:explicit-stability-exponent}
\end{align}

\begin{theorem}[Stability of the moment map]
\label{thm:quantitative-torelli}
There exists \(c_{d,m}>0\) such that, for every \(P,Q\in V_{d,m}\),
\begin{equation}
 \|\mathbf J_{d,m}(P)-\mathbf J_{d,m}(Q)\|
 \geq
 c_{d,m}
 \left(
  \frac{\delta_{\mathrm{orb}}(P,Q)}
  {\sqrt2(1+\|P\|_V^2+\|Q\|_V^2)}
 \right)^{A_{d,m}}.
 \label{eq:global-orbital-Lojasiewicz}
\end{equation}
For every \(R>0\) and all \(P,Q\in V_{d,m}\) with
\(\|P\|_V,\|Q\|_V\leq R\),
\begin{equation}
 \delta_{\mathrm{orb}}(P,Q)
 \leq
 \sqrt2(1+2R^2)c_{d,m}^{-1/A_{d,m}}
 \mathfrak E_{\mathrm{probe}}(P,Q)^{1/(2A_{d,m})}.
 \label{eq:ball-orbital-Holder}
\end{equation}
For every \(R>0\), there is a constant \(L_{d,m,R}\) such that,
for all \(P,Q\in V_{d,m}\) with \(\|P\|_V,\|Q\|_V\leq R\),
\begin{equation}
 \sqrt{\mathfrak E_{\mathrm{probe}}(P,Q)}
 \leq
 L_{d,m,R}\delta_{\mathrm{orb}}(P,Q).
 \label{eq:probe-upper-Lipschitz}
\end{equation}
\end{theorem}

\begin{proof}
Consider the polynomial map
\[
 \mathcal F:
 V_{d,m}\times V_{d,m}
 \longrightarrow
 \R^{N_{\mathrm{probe}}(d,m)},
 \qquad
 \mathcal F(P,Q)
 =\mathbf J_{d,m}(P)-\mathbf J_{d,m}(Q).
\]
The product space is equipped with the Euclidean product of the tensor
norm \eqref{eq:invariant-tensor-norm}.
Every coordinate \(M_{P,r}(\Sigma)\) is homogeneous of degree \(r\) in
the coefficients of \(P\), so
\[
 \deg\mathcal F\leq D.
\]
By \cref{thm:intro-torelli}, its real zero set is exactly
\begin{equation}
 Z_{\mathrm{orb}}
 =
 \{(P,Q):Q=P\circ U\text{ for some }U\in O(d)\}.
 \label{eq:orbital-incidence-zero-set}
\end{equation}

Since \((0,0)\in Z_{\mathrm{orb}}\),
\(\operatorname{dist}(z,Z_{\mathrm{orb}})/(1+\|z\|^2)\leq1\).
The global polynomial \L ojasiewicz inequality
\cite[Corollary~10]{KurdykaSpodzieja}, applied in dimension \(2N\) and with
the degree bounded by \(D\), therefore gives a constant
\(c_{d,m}>0\) such that
\[
 \|\mathcal F(z)\|
 \geq
 c_{d,m}
 \left(
  \frac{\operatorname{dist}(z,Z_{\mathrm{orb}})}
       {1+\|z\|^2}
 \right)^{D(6D-3)^{2N-1}}.
\]
The distance to \(Z_{\mathrm{orb}}\) satisfies
\begin{equation}
 \frac1{\sqrt2}\delta_{\mathrm{orb}}(P,Q)
 \leq
 \operatorname{dist}\bigl((P,Q),Z_{\mathrm{orb}}\bigr)
 \leq
 \delta_{\mathrm{orb}}(P,Q).
 \label{eq:two-orbit-distances}
\end{equation}
The upper bound follows by comparing \((P,Q)\) with
\((P,P\circ U)\).  For the lower bound, let
\((P',Q')\in Z_{\mathrm{orb}}\), say \(Q'=P'\circ U\).  Since
\eqref{eq:invariant-tensor-norm} is \(O(d)\)-invariant,
\begin{align*}
 \delta_{\mathrm{orb}}(P,Q)
 &\leq\|Q-P\circ U\|_V\\
 &\leq\|Q-Q'\|_V+\|(P'-P)\circ U\|_V\\
 &\leq\sqrt2\,
 \|(P,Q)-(P',Q')\|.
\end{align*}
Taking the infimum proves \eqref{eq:two-orbit-distances}, and substitution
proves \eqref{eq:global-orbital-Lojasiewicz}.  Since
\[
 \mathfrak E_{\mathrm{probe}}(P,Q)
 =\|\mathcal F(P,Q)\|^2,
\]
\eqref{eq:ball-orbital-Holder} follows.

Finally, on the radius-\(R\) ball the derivative of the polynomial map
\(\mathbf J_{d,m}\) is bounded.  For a minimising \(U\in O(d)\), invariance
of \(\mathbf J_{d,m}\) and the mean-value theorem give
\[
 \|\mathbf J_{d,m}(Q)-\mathbf J_{d,m}(P)\|
 =
 \|\mathbf J_{d,m}(Q)-\mathbf J_{d,m}(P\circ U)\|
 \leq L_{d,m,R}\|Q-P\circ U\|_V,
\]
which is \eqref{eq:probe-upper-Lipschitz}.
\end{proof}

Thus \(\mathbf J_{d,m}\) induces a bi-H\"older embedding of each
coefficient ball modulo \(O(d)\).  The proof does not give a numerical value of \(c_{d,m}\).
By \cref{thm:intro-period-realisation}, the coordinates of this map
are coefficients of the local formal period at the origin.

\begin{remark}[Failure of Lipschitz continuity for the inverse map at zero]
\label{rem:no-global-Lipschitz}
For \(v\in\R^d\), set \(Q_v(x)=\langle v,x\rangle\). Then
\[
 \delta_{\mathrm{orb}}(0,Q_v)=|v|,
 \qquad
 \|\mathbf J_{d,m}(Q_v)-\mathbf J_{d,m}(0)\|
 \asymp |v|^2
 \qquad(|v|\to0).
\]
Thus the inverse moment map is not Lipschitz near the zero orbit.
On each regular quotient stratum, a generating Hilbert map is a smooth
embedding \cite{Schwarz,ProcesiSchwarz}.  Since finite differences recover
such a map from \(\mathbf J_{d,m}\), the inverse moment map is locally
Lipschitz there.  The example above precludes a uniform bound across all
strata.
\end{remark}

\begin{remark}[The real Gaussian contour]
\label{rem:observable-cyclicity}
Orbit separation uses the moment vector for several Gaussian covariance
matrices.  It does not require a single real contour to generate the full
rapid-decay homology under continuation.  For an even polynomial,
integration over the real Gaussian contour vanishes on odd integrands.
For the quartic, \cref{prop:quartic-parity-rank} in
Appendix~\ref{sec:quartic} computes a two-dimensional monodromy orbit
inside a rank-three system. The corresponding monomial moment-jet ranks
are computed in \cref{prop:monomial-physical-rank} in
Appendix~\ref{sec:one-variable-refinements}.
\end{remark}


\section{Quadratic twisted de Rham cohomology}
\label{sec:quadratic-comparison}

Let \(R\) be a commutative \(\mathbf Q\)-algebra, let \(V=R^d\), and let
\[
 B\in\Sym_d(R)
\]
be invertible.  Put \(C=B^{-1}\), write
\(x=(x_1,\ldots,x_d)\), and set
\[
 Q_B(x)=\frac12x^{\mathsf T}Bx.
\]
On
\[
 \Omega_R^\bullet
 =
 R[x_1,\ldots,x_d]\otimes_R\bigwedge^\bullet V^*
\]
consider
\begin{align}
 \delta_B\alpha
 &=
 \sum_{j=1}^d(Bx)_j\,\mathrm d x_j\wedge\alpha,
 \label{eq:quadratic-forest-differential}\\
 \nabla_B\alpha
 &=
 d_x\alpha-d_xQ_B\wedge\alpha
 =
 \sum_{j=1}^d
 \mathrm d x_j\wedge
 \bigl(\partial_{x_j}-(Bx)_j\bigr)\alpha.
 \label{eq:quadratic-twisted-dR}
\end{align}
The first is the Koszul differential of the sequence
\((Bx)_1,\ldots,(Bx)_d\); the second is the relative de Rham differential of the
exponential module \(\mathcal E^{-Q_B}\).  Define
\[
 \Delta_C=\sum_{a,b=1}^dC_{ab}\partial_{x_a}\partial_{x_b},
 \qquad
 \mathcal W_C=\exp\!\left(-\frac12\Delta_C\right).
\]
The series \(\exp(\pm\Delta_C/2)\) terminate on every polynomial.

\begin{theorem}[Quadratic conjugation]
\label{thm:quadratic-forest-dR}
For every \(p\), define on \(\Omega_R^p\)
\[
 \Theta_B^p=(-1)^p\mathcal W_C.
\]
Here \(\mathcal W_C\) acts on the polynomial coefficients. Then
\[
 \Theta_B^{p+1}\delta_B=\nabla_B\Theta_B^p.
\]
Consequently
\[
 \Theta_B:
 (\Omega_R^\bullet,\delta_B)
 \xrightarrow{\ \simeq\ }
 (\Omega_R^\bullet,\nabla_B)
\]
is an isomorphism of complexes, with inverse
\[
 (\Theta_B^p)^{-1}
 =
 (-1)^p\exp\!\left(\frac12\Delta_C\right).
\]
Moreover,
\[
 H^p(\Omega_R^\bullet,\nabla_B)=0\quad(p\neq d),
 \qquad
 H^d(\Omega_R^\bullet,\nabla_B)\simeq R.
\]
Let \(\operatorname{vol}_x=\mathrm d x_1\wedge\cdots\wedge\mathrm d x_d\).
For every \(f\in R[x_1,\ldots,x_d]\),
\begin{equation}
 [f\,\operatorname{vol}_x]_{\nabla_B}
 =
 \left[
 \exp\!\left(\frac12\Delta_C\right)f
 \right]_{x=0}
 [\operatorname{vol}_x]_{\nabla_B}.
\label{eq:cohomological-wick-rule}
\end{equation}
\end{theorem}

\begin{proof}
Since \(C=B^{-1}\),
\[
 [\Delta_C,x_\ell]
 =
 2\sum_{a=1}^dC_{\ell a}\partial_{x_a},
 \qquad
 [\Delta_C,(Bx)_j]=2\partial_{x_j}.
\]
All higher iterated commutators vanish.  The
Baker--Campbell--Hausdorff identity therefore terminates and gives
\begin{equation}
 \mathcal W_C(Bx)_j\mathcal W_C^{-1}
 =
 (Bx)_j-\partial_{x_j}.
\label{eq:wick-conjugation}
\end{equation}
Since \(\mathcal W_C\) commutes with exterior multiplication,
\[
 \mathcal W_C\delta_B\mathcal W_C^{-1}
 =
 \sum_j\mathrm d x_j\wedge
 \bigl((Bx)_j-\partial_{x_j}\bigr)
 =-\nabla_B.
\]
Multiplication by \((-1)^p\) on degree \(p\) removes this sign, proving the
chain identity.  Since \(\Delta_C\) lowers polynomial degree by two, the
two exponential operators are mutually inverse on polynomials.

The sequence \((Bx)_1,\ldots,(Bx)_d\) is an invertible linear transform of
\(x_1,\ldots,x_d\), hence a regular sequence.  Its Koszul cohomology \cite[Tag 062F]{StacksProject} is
concentrated in degree \(d\), where it equals
\[
 R[x]/((Bx)_1,\ldots,(Bx)_d)\simeq R.
\]
The chain isomorphism transfers this statement to \(\nabla_B\).  Finally,
apply \(\Theta_B^{-1}\) to \(f\,\operatorname{vol}_x\) and pass to the
Koszul quotient, which is evaluation at \(x=0\).  This gives
\eqref{eq:cohomological-wick-rule}.
\end{proof}

The Wick expansion \cite{PeccatiTaqqu} can be written in terms of partial
matchings.  For a monomial
\[
 x_{\boldsymbol i}=x_{i_1}\cdots x_{i_n}
\]
and a partial matching \(M\) of \(\{1,\ldots,n\}\), let
\(V(M)=\bigcup_{\{a,b\}\in M}\{a,b\}\) be the set of matched vertices
and let \(|M|\) be the number of edges. Put
\[
 \Contr_{C,M}(x_{\boldsymbol i})
 =
 \left(\prod_{\{a,b\}\in M}C_{i_ai_b}\right)
 \prod_{a\notin V(M)}x_{i_a}.
\]
Then
\begin{equation}
 \exp\!\left(\pm\frac12\Delta_C\right)x_{\boldsymbol i}
 =
 \sum_M(\pm1)^{|M|}
 \Contr_{C,M}(x_{\boldsymbol i}).
\label{eq:partial-matching-expansion}
\end{equation}

Indeed, in the term
\(2^{-k}(k!)^{-1}\Delta_C^k\), every matching with \(k\) edges occurs
with \(2^kk!\) orderings and orientations.

\begin{corollary}[Quadratic exponential direct image]
\label{cor:quadratic-direct-image}
Put \(q(x)=\sum_{j=1}^d x_j^2\).  Let \(A\in\Sym_d(\R)\) and
\[
 B(t)=I_d-2itA,
 \qquad C(t)=B(t)^{-1}.
\]
Over \(\C[t,\det B(t)^{-1}]\), the relative exponential direct image of
\[
 \exp\!\left(-\frac{q(x)}{2}+itx^{\mathsf T}Ax\right)
\]
has rank one and is concentrated in relative de Rham degree \(d\).  Its
normalised period over the real Gaussian contour is
\[
 \Phi_A(t)=\det(I_d-2itA)^{-1/2},
 \qquad \Phi_A(0)=1,
\]
and satisfies
\[
 \left(
 \partial_t-i\,\operatorname{tr}(AC(t))
 \right)\Phi_A(t)=0.
\]
In particular, the real Gaussian cycle is cyclic in this rank-one direct
image.
\end{corollary}

\begin{proof}
Differentiation of the exponential gives
\[
 \nabla_{\partial_t}^{\mathrm{GM}}
 [\operatorname{vol}_x]
 =
 \left[
 -\frac12x^{\mathsf T}B'(t)x\,\operatorname{vol}_x
 \right].
\]
By \eqref{eq:cohomological-wick-rule},
\[
 [x^{\mathsf T}B'(t)x\,\operatorname{vol}_x]
 =
 \operatorname{tr}(C(t)B'(t))
 [\operatorname{vol}_x].
\]
Since \(B'(t)=-2iA\), the differential equation follows.  Direct Gaussian
integration gives the determinant formula, which is nonzero wherever its branch with \(\Phi_A(0)=1\) is defined.
\end{proof}

\begin{remark}[Correlated quadratic copies]
\label{rem:replicated-quadratic}
For \(r\) correlated replicas with covariance \(\Sigma\) and
\(P(x)=x^{\mathsf T}Ax\), take
\[
 B(\boldsymbol t)
 =
 \Sigma^{-1}\otimes I_d
 -
 2i\,\operatorname{diag}(t_1,\ldots,t_r)\otimes A.
\]
Theorem~\ref{thm:quadratic-forest-dR} applies in \(rd\) variables over the
ring obtained by inverting \(\det B(\boldsymbol t)\).  Its Wick contractions use the
covariance matrix \(B(\boldsymbol t)^{-1}\).
\end{remark}

\begin{example}[A cubic phase]
\label{prop:quadratic-comparison-boundary}
For a polynomial \(f\), let
\(M_f\) denote multiplication by \(f\).  Then
\[
 \left[-\frac12\Delta_C,M_f\right]
 =
 -\frac12M_{\Delta_Cf}
 -
 \sum_{a,b}C_{ab}M_{\partial_af}\partial_{x_b}.
\]
For linear \(f\), higher commutators vanish; for quadratic \(f\), they
need not vanish.

For \(P=H_3=x^3-3x\), the top twisted de Rham quotient over \(\C(t)\) is
\[
 \cH_{H_3}
 =
 \frac{\C(t)[x]}
 {D_t\C(t)[x]},
 \qquad
 D_t=\partial_x-x+it(3x^2-3),
\]
and has basis \([1],[x]\).  Indeed,
\[
 D_t(x^n)
 =
 nx^{n-1}-x^{n+1}+it(3x^{n+2}-3x^n),
\]
so every \(x^{n+2}\) reduces to lower degree and \([1],[x]\) span.
Conversely, if \(f\neq0\) has degree \(n\), then \(D_tf\) has degree
\(n+2\), with leading coefficient \(3it\) times that of \(f\).  Hence no
nonzero polynomial of degree at most one belongs to the image, and the two classes are
independent.  The quotient has dimension two, whereas the one-variable quadratic Koszul complex
has top cohomology of dimension one.  Thus the two complexes are not
quasi-isomorphic. The exact periods of this cubic example, their
second-order equation, and their Stokes normalisation are computed in
\cref{thm:h3-airy-stokes} in Appendix~\ref{sec:cubic}.
\end{example}

\section{Reduction to the Jacobian algebra}
\label{sec:global-comparison}

\subsection{Reduction over a field}

Let \(k\) be a field of characteristic zero and let
\[
 F=F_m+F_{m-1}+\cdots+F_0\in k[x_1,\ldots,x_d],
 \qquad m\geq3,
\]
with each \(F_j\) homogeneous of degree \(j\).  Assume that, over an
algebraic closure of \(k\),
\begin{equation}
 \partial_1F_m(x)=\cdots=\partial_dF_m(x)=0
 \quad\Longrightarrow\quad x=0.
 \label{eq:leading-gradient-isolated}
\end{equation}
Put
\[
 \mathcal J(F_m)=k[x]/(\partial_1F_m,\ldots,\partial_dF_m),
 \qquad
 \Omega_k^\bullet=k[x]\otimes_k\bigwedge^\bullet(k^d)^*,
 \qquad
 \nabla_F=\mathrm d_x+\mathrm d_xF\wedge.
\]
Choose homogeneous polynomial representatives for \(\mathcal J(F_m)\),
and let \(\iota\) send each representative \(e\) to
\(e\operatorname{vol}_x\).

\begin{theorem}[Jacobian reduction]
\label{thm:jacobi-normal-form-transfer}
The cohomology satisfies
\begin{equation}
 H^p(\Omega_k^\bullet,\nabla_F)=0\quad(p\neq d),
 \qquad
 \iota:\mathcal J(F_m)\xrightarrow{\ \simeq\ }
 H^d(\Omega_k^\bullet,\nabla_F).
 \label{eq:exact-jacobi-rank}
\end{equation}
In particular,
\begin{equation}
 \mu=\dim_k\mathcal J(F_m)=(m-1)^d.
 \label{eq:milnor-number}
\end{equation}
\end{theorem}

\begin{proof}
The partial derivatives of \(F_m\) form a homogeneous system of
parameters in \(k[x]\), hence a regular sequence
\cite[Tags 00N7 and 062F]{StacksProject}.  Thus the Koszul differential
\(\delta_0=\mathrm dF_m\wedge\) has cohomology only in degree \(d\),
equal to \(\mathcal J(F_m)\operatorname{vol}_x\).  The Hilbert series
\[
 \frac{(1-z^{m-1})^d}{(1-z)^d}
 =(1+z+\cdots+z^{m-2})^d
\]
gives \eqref{eq:milnor-number}.

Give \(f\,\mathrm d x_I\) weight
\(\deg_x f-(m-1)|I|\).  Each weight subcomplex is finite dimensional.
Choose a decomposition
\(\Omega^p=B^p\oplus H^p\oplus L^p\) in each weight, with
\(\ker\delta_0=B^p\oplus H^p\) and
\(\delta_0:L^p\xrightarrow{\sim}B^{p+1}\), taking
\(H^d=\iota(\mathcal J(F_m))\) and \(H^p=0\) for \(p\neq d\).
Let \(h\) be the inverse on \(B^{p+1}\), zero on the other summands,
and let \(\pi\) be projection onto \(H^d\).  Then
\begin{equation}
 \delta_0h+h\delta_0=\Id-\iota\pi,
 \qquad \pi\iota=\Id,
 \qquad h^2=\pi h=h\iota=0.
 \label{eq:koszul-contraction}
\end{equation}
The map \(h\) lowers polynomial degree by \(m-1\), whereas
\[
 \eta=\nabla_F-\delta_0
 =\mathrm d_x+\mathrm d(F-F_m)\wedge
\]
raises it by at most \(m-2\).  Hence \(\eta h\) and \(h\eta\)
lower degree by at least one.  The basic perturbation lemma
\cite{GugenheimLambe} gives
\begin{align}
 \pi_F&=\pi(1+\eta h)^{-1}
       =\pi\sum_{j\geq0}(-\eta h)^j,
 \label{eq:perturbed-projection}\\
 h_F&=h(1+\eta h)^{-1},
 \label{eq:perturbed-homotopy}\\
 \nabla_Fh_F+h_F\nabla_F&=\Id-\iota\pi_F,
 \qquad \pi_F\nabla_F=0.
 \label{eq:perturbed-contraction}
\end{align}
On forms whose polynomial coefficients have degree at most \(N\),
the sums have at most \(N+1\) nonzero terms.  Since \(\eta\iota=0\), the inclusion remains \(\iota\);
the transferred differential is zero because its target is concentrated
in degree \(d\).  Hence \(\pi_F\) induces the inverse of \(\iota\),
and every polynomial \(f\) satisfies
\begin{equation}
 [f\operatorname{vol}_x]_{\nabla_F}
 =
 [\iota\pi_F(f\operatorname{vol}_x)]_{\nabla_F}.
 \label{eq:exact-jacobi-normal-form}
\end{equation}
\end{proof}

For the Gaussian phase, take \(k=K=\C(t)\),
\(P=P_m+\cdots+P_0\in\C[x]\), and
\[
 F_{P,t}=-\frac{q(x)}2+itP(x),
 \qquad
 \mathcal J_m(P)=\C[x]/(\partial_1P_m,\ldots,\partial_dP_m).
\]
Here the hypothesis is
\begin{equation}
 \partial_1P_m(x)=\cdots=\partial_dP_m(x)=0,
 \quad x\in\C^d
 \quad\Longrightarrow\quad x=0.
 \label{eq:principal-isolated}
\end{equation}
The differential and its decomposition are
\begin{align}
 \nabla_{P,t}&=\mathrm d_x+\mathrm dF_{P,t}\wedge,
 \label{eq:global-twisted-differential}\\
 \delta_{0,t}&=it\,\mathrm dP_m\wedge,
 \label{eq:principal-koszul}\\
 \eta_{P,t}&=\mathrm d_x-\sum_jx_j\,\mathrm d x_j\wedge
             +it\,\mathrm d(P-P_m)\wedge.
 \label{eq:lower-perturbation}
\end{align}
If \(h\) contracts the Koszul complex of \(\mathrm dP_m\wedge\), then
\(h_t=(it)^{-1}h\) contracts that of \(\delta_{0,t}\).  We write
\(\pi_{P,t}\) for the reduction \(\pi_{F_{P,t}}\).
For related constructions with Brieskorn modules and twisted de Rham
complexes, see \cite{Brieskorn,SabbahTwisted}.

Let \(S\) be an irreducible algebraic stratum of degree-\(m\)
polynomials whose leading gradient is isolated, and set
\(K_S=\C(S)(t)\).  Assume that fixed monomials
\(e_1,\ldots,e_\mu\) form a basis of the leading Jacobian algebra
throughout \(S\).  For the generic polynomial \(P\) on \(S\), put
\[
 \mathcal J_{m,S}=\C(S)[x]/(\partial_1P_m,\ldots,\partial_dP_m).
\]

\begin{corollary}[Reduction over the parameter field]
\label{cor:jacobi-comparison-determinant}
The isomorphism
\[
 \iota:\mathcal J_{m,S}\otimes_{\C(S)}K_S
 \xrightarrow{\ \simeq\ }
 H^d(\Omega_{K_S}^\bullet,\nabla_{P,t}),
 \qquad e_j\longmapsto[e_j\operatorname{vol}_x],
\]
and its inverse, induced by \(\pi_{P,t}\), are defined over \(K_S\).
\end{corollary}

\begin{proof}
Choose the complements in \eqref{eq:koszul-contraction} by Gaussian
elimination over \(\C(S)\).  Then \(h_t\) and the finite sum
\eqref{eq:perturbed-projection} are defined over \(K_S\).
Apply \cref{thm:jacobi-normal-form-transfer} over this field.
\end{proof}

In a chosen rational de Rham frame, the matrix of \(\iota\) is denoted by
\begin{equation}
 B_{\mathrm{Jac}}(P,t)\in\operatorname{GL}_\mu(K_S).
 \label{eq:B-Jacobi}
\end{equation}
For finitely many insertions, their reduction coefficients and the
matrix entries are regular where their denominators do not vanish.

Fix \(P\in S\), put \(K=\C(t)\), and choose a monomial order.
Let \(b_1,\ldots,b_\mu\) be the standard
monomials for the ideal \((\partial_1P_m,\ldots,\partial_dP_m)\), ordered by increasing degree.  They
span a space \(E\subset\C[x]\).  Put
\[
 \Delta=\sum_j\partial_{x_j}^2,
 \qquad \mathcal W=\exp(-\Delta/2).
\]

\begin{corollary}[Two polynomial bases]
\label{cor:unitriangular-wick-jacobi}
The two families
\[
 ([b_j\operatorname{vol}_x])_{j=1}^\mu,
 \qquad
 ([\mathcal Wb_j\operatorname{vol}_x])_{j=1}^\mu
\]
are bases of \(H^d(\Omega_K^\bullet,\nabla_{P,t})\).
Let \(B_{\mathrm{Jac}}^{\mathrm{std}}\) be the matrix whose columns
express the second basis in the first.
Then
\begin{equation}
 B_{\mathrm{Jac}}^{\mathrm{std}}\text{ is triangular},
 \qquad
 \det B_{\mathrm{Jac}}^{\mathrm{std}}=1,
 \label{eq:standard-Jacobi-determinant}
\end{equation}
and its inverse is the matrix of \(\exp(\Delta/2)|_E\) in the basis
\((b_1,\ldots,b_\mu)\).
\end{corollary}

\begin{proof}
Every monomial in \(\Delta^kb_j\) divides \(b_j\), so it is standard.
Thus \(\Delta E\subset E\), and \(\mathcal Wb_j-b_j\) has smaller
degree than \(b_j\).  Both exponential operators preserve \(E\) and
are inverse.  Apply \cref{thm:jacobi-normal-form-transfer} to \(E\).
\end{proof}

The matrix \(B_{\mathrm{Jac}}^{\mathrm{std}}\) compares these two
polynomial bases; \(B_{\mathrm{Jac}}\) in \eqref{eq:B-Jacobi} compares
the Jacobian representatives with the chosen de Rham frame.
The algebraic residue pairing has its own Gram determinant:
\cref{prop:residue-determinant} in
Appendix~\ref{sec:one-variable-refinements} computes it explicitly in
one variable, including at multiple roots. This is a separate pairing
calculation from the change-of-basis determinant above.

\subsection{Correlated copies}

Fix \(\Sigma\in\operatorname{Sym}_r^{++}(\R)\) and
\(P\in\C[x_1,\ldots,x_d]\) with isolated leading gradient.  Work over
\(K_r=\C(t_1,\ldots,t_r)\), with
\(\boldsymbol x=(x_1,\ldots,x_r)\in(\C^d)^r\), and define
\[
 F_{\Sigma,\boldsymbol t}(\boldsymbol x)
 =-\frac12\sum_{a,b=1}^r(\Sigma^{-1})_{ab}\langle x_a,x_b\rangle
   +i\sum_{a=1}^rt_aP(x_a),
 \qquad
 \nabla_{\Sigma,\boldsymbol t}
 =\mathrm d_{\boldsymbol x}
   +\mathrm d_{\boldsymbol x}F_{\Sigma,\boldsymbol t}\wedge.
\]

\begin{corollary}[Jacobian reduction for correlated copies]
\label{cor:replicated-jacobi-transfer}
The cohomology satisfies
\[
 H^q(\Omega_{K_r}^\bullet,\nabla_{\Sigma,\boldsymbol t})=0
 \quad(q\neq rd),
 \qquad
 H^{rd}(\Omega_{K_r}^\bullet,\nabla_{\Sigma,\boldsymbol t})
 \simeq\mathcal J_m(P)^{\otimes r}\otimes_\C K_r.
\]
The isomorphism is induced by tensor products of homogeneous
representatives of \(\mathcal J_m(P)\).  In particular, the dimension
of the top cohomology is \(\mu_r=(m-1)^{rd}\).
The same conclusion holds at fixed \(\boldsymbol t\) with
\(t_1\cdots t_r\neq0\).
\end{corollary}

\begin{proof}
The leading gradient consists of
\((it_a\partial_jP_m(x_a))_{a,j}\).  These are regular sequences in
disjoint variables, and their quotient is
\(\mathcal J_m(P)^{\otimes r}\otimes_\C K_r\).
The quadratic term has degree less than \(m\), so
\cref{thm:jacobi-normal-form-transfer} applies in \(rd\) variables.
\end{proof}

If \(P\) and \(\Sigma\) vary, use the function field of their parameter
space in place of \(\C\).  On a stratum with fixed standard monomials,
\cref{cor:jacobi-comparison-determinant} gives reduction coefficients
rational in \(\Sigma\), \(\boldsymbol t\), and the coefficients of
\(P\), for every polynomial insertion, including
\(\prod_{a=1}^rP(x_a)\).

\section{Vanishing coupling parameters}
\label{sec:zero-coupling}
\label{subsec:gaussian-face-cube}

Let \(R=\{1,\ldots,r\}\), and let \(\Sigma_K\) denote the principal
submatrix indexed by \(K\subset R\).  For \(S\subset K\), set
\[
 \mathbf k_S=\C(t_a:a\in S),\qquad \mathbf k_\varnothing=\C,
\]
and define
\begin{equation}
 F_{K\mid S}(\boldsymbol x_K)
 =-\frac12\sum_{a,b\in K}
 (\Sigma_K^{-1})_{ab}\langle x_a,x_b\rangle
 +i\sum_{a\in S}t_aP(x_a).
 \label{eq:face-phase}
\end{equation}
The complex at vanishing couplings is
\[
 \mathcal C_{K\mid S}
 =\left(\Omega^\bullet_{\mathbf k_S[\boldsymbol x_K]},
 \mathrm d+\mathrm dF_{K\mid S}\wedge\right).
\]
Thus the inactive couplings are set to zero in the differential before
cohomology is taken.

For \(S\subset I\subset K\), write \(J=K\setminus I\) and
\begin{equation}
 L_{J\mid I}=\Sigma_{JI}\Sigma_{II}^{-1},\qquad
 C_{J\mid I}=\Sigma_{JJ}-\Sigma_{JI}\Sigma_{II}^{-1}\Sigma_{IJ}.
 \label{eq:conditional-blocks}
\end{equation}
Define the polynomial Gaussian integration operator
\begin{equation}
 \mathsf W^\Sigma_{K\to I}f(\boldsymbol x_I)
 =\left.
 \exp\!\left(\frac12\Delta_{C_{J\mid I}}^{(\boldsymbol y)}\right)
 f(\boldsymbol x_I,L_{J\mid I}\boldsymbol x_I+\boldsymbol y_J)
 \right|_{\boldsymbol y_J=0},
 \label{eq:conditional-Wick-map}
\end{equation}
where
\begin{equation}
 \Delta_C^{(\boldsymbol y)}
 =\sum_{a,b\in J}C_{ab}\sum_{\nu=1}^d
 \partial_{y_{a,\nu}}\partial_{y_{b,\nu}}.
 \label{eq:conditional-Laplacian}
\end{equation}
The exponential is a finite sum on every polynomial.  We use the determinant
orientation lines \(\mathfrak o_A=\det((\R^d)^A)\), with their associative
identifications \(\mathfrak o_K\simeq\mathfrak o_I\otimes
\mathfrak o_{K/I}\).  For \(I=\varnothing\), set
\(L_{K\mid\varnothing}=0\), \(C_{K\mid\varnothing}=\Sigma_K\), and
\(\operatorname{vol}_\varnothing=1\).

\begin{theorem}[Gaussian integration of inactive variables]
\label{thm:gaussian-face-cube}
Assume \eqref{eq:principal-isolated}, and put \(\mu=(m-1)^d\).
For \(S\subset I\subset K\), the map
\begin{equation}
 \mathbf W^{\Sigma,S}_{K\to I}:
 H^{|K|d}(\mathcal C_{K\mid S})
 \xrightarrow{\ \simeq\ }H^{|I|d}(\mathcal C_{I\mid S}),
 \qquad
 [f\operatorname{vol}_K]\longmapsto
 [\mathsf W^\Sigma_{K\to I}f\operatorname{vol}_I]
 \label{eq:face-cohomology-map}
\end{equation}
is an isomorphism under these orientation identifications.  Moreover,
\begin{equation}
 H^q(\mathcal C_{K\mid S})=0\quad(q\neq |K|d),\qquad
 \dim_{\mathbf k_S}H^{|K|d}(\mathcal C_{K\mid S})=\mu^{|S|}.
 \label{eq:face-rank}
\end{equation}
\end{theorem}

\begin{proof}
The shear \(\boldsymbol y_J=\boldsymbol x_J-L_{J\mid I}\boldsymbol x_I\)
has determinant one.  The Schur complement identities give
\begin{equation}
 F_{K\mid S}(\boldsymbol x_I,
 L_{J\mid I}\boldsymbol x_I+\boldsymbol y_J)
 =F_{I\mid S}(\boldsymbol x_I)
 -\frac12\langle\boldsymbol y_J,
 (C_{J\mid I}^{-1}\otimes I_d)\boldsymbol y_J\rangle.
 \label{eq:face-Schur-splitting}
\end{equation}
Hence \(\mathcal C_{K\mid S}\simeq
\mathcal C_{I\mid S}\otimes\mathcal G_{C_{J\mid I}}\).
By \cref{thm:quadratic-forest-dR}, the quadratic factor has
one-dimensional top cohomology and projection
\(h\mapsto(e^{\Delta_{C_{J\mid I}}/2}h)(0)\).
Tensoring this projection with the identity gives
\eqref{eq:face-cohomology-map}.  With \(I=S\), the active factor has
cohomology \(\mathcal J_m(P)^{\otimes |S|}\otimes\mathbf k_S\) in degree
\(|S|d\), by \cref{cor:replicated-jacobi-transfer}.  This proves
\eqref{eq:face-rank}.
\end{proof}

The identity
\(\mathsf W^\Sigma_{K\to I}f=\E[f(X_K)\mid X_I]\) gives, for
\(S\subset I\subset K\subset L\),
\begin{equation}
 \mathsf W^{\Sigma_K}_{K\to I}\circ
 \mathsf W^{\Sigma_L}_{L\to K}
 =\mathsf W^{\Sigma_L}_{L\to I}.
 \label{eq:conditional-Wick-tower}
\end{equation}
The induced cohomology maps satisfy the same composition law.
If the coordinate volume forms are ordered independently, these maps
include the corresponding shuffle signs.

Put
\[
 \Phi_{P,K,\Sigma_K}(\boldsymbol\tau_K)
 =\E\exp\!\left(i\sum_{a\in K}\tau_aP(X_a)\right).
\]

\begin{corollary}[Derivatives at vanishing couplings]
\label{cor:normal-jets-face-cube}
For \(I\subset K\) and \(\alpha\in\N^{K\setminus I}\),
\begin{align}
 &i^{-|\alpha|}
 \partial_{\boldsymbol\tau_{K\setminus I}}^\alpha
 \Phi_{P,K,\Sigma_K}(\boldsymbol\tau_I,0)\notag\\
 &\qquad=
 \E_{\Sigma_{II}}\!\left[
 \mathsf W^{\Sigma_K}_{K\to I}
 \left(\prod_{b\in K\setminus I}P(x_b)^{\alpha_b}\right)(X_I)
 e^{i\sum_{a\in I}\tau_aP(X_a)}\right].
 \label{eq:all-face-normal-jet}
\end{align}
In particular,
\begin{equation}
 M_{P,r}(\Sigma)
 =\mathsf W^\Sigma_{R\to\varnothing}
 \left(\prod_{a=1}^rP(x_a)\right).
 \label{eq:Torelli-deepest-face}
\end{equation}
\end{corollary}

\begin{proof}
The derivatives are dominated by Gaussian-integrable polynomials.
Differentiate under the expectation and condition on \(X_I\).
For \eqref{eq:Torelli-deepest-face}, take \(I=\varnothing\), \(K=R\),
and \(\alpha_b=1\).
\end{proof}

\subsection{Radial scaling}
\label{subsec:toric-Rees-Wick}

Fix \(S\subset K\subset R\).  Set
\[
 A_S=\C[\lambda_a,\lambda_a^{-1}:a\in S],\qquad
 A_\varnothing=\C,\qquad \lambda_a=0\ (a\notin S),\qquad N_K=d|K|.
\]
For \(0\leq j\leq m\), put
\(H_{S,j}(\boldsymbol u)=i\sum_{a\in S}\lambda_aP_{m-j}(u_a)\), and define
\begin{align}
 G_{K\mid S}(\boldsymbol u)
 &=-\frac12\sum_{a,b\in K}
 (\Sigma_K^{-1})_{ab}\langle u_a,u_b\rangle+H_{S,0}(\boldsymbol u),
 \label{eq:toric-leading-phase}\\
 \Psi_{K\mid S}(\rho,\boldsymbol u)
 &=G_{K\mid S}(\boldsymbol u)
 +\sum_{j=1}^{m}\rho^jH_{S,j}(\boldsymbol u).
 \label{eq:toric-Psi}
\end{align}
The substitution \(t_a=\rho^{m-2}\lambda_a\), \(x_a=\rho^{-1}u_a\) gives
\begin{equation}
 F_{\Sigma_K,\boldsymbol t}(\rho^{-1}\boldsymbol u)
 =\rho^{-2}\Psi_{K\mid S}(\rho,\boldsymbol u).
 \label{eq:toric-rescaled-phase}
\end{equation}
On polynomial forms completed in \(\rho\), consider
\begin{equation}
 \mathfrak D_{K\mid S}
 =\rho^2\mathrm d_{\boldsymbol u}
 +\mathrm d_{\boldsymbol u}\Psi_{K\mid S}\wedge.
 \label{eq:toric-Rees-differential}
\end{equation}
The normalisation \(\mathrm d\boldsymbol x_K
=\rho^{-N_K}\mathrm d\boldsymbol u_K\) makes the pure Gaussian period
constant.  The corresponding radial operator is
\begin{equation}
 \nabla_{\rho,K}^{\mathrm{phys}}
 =\partial_\rho+\partial_\rho(\rho^{-2}\Psi_{K\mid S})
 -\frac{N_K}{\rho}.
 \label{eq:physical-radial-connection}
\end{equation}

\begin{lemma}[Jacobian algebra over the parameter ring]
\label{lem:radial-Jacobian}
Assume \eqref{eq:principal-isolated}.  Over \(A_S\), the Koszul complex
of \(\mathrm dG_{K\mid S}\) contracts onto
\[
 V_S\otimes_\C A_S[-N_K],\qquad
 V_S=\bigotimes_{a\in S}\mathcal J_m(P).
\]
In particular, \(\mathcal J(G_{K\mid S})\) is free over \(A_S\), and each
specialisation with \(\lambda_a\neq0\) for \(a\in S\) satisfies
\begin{equation}
 \dim_\C\mathcal J(G_{K\mid S})=\mu^{|S|}.
 \label{eq:toric-face-Jacobi-rank}
\end{equation}
\end{lemma}

\begin{proof}
For \(J=K\setminus S\), the Schur shear gives
\[
 G_{K\mid S}=G_{S\mid S}(\boldsymbol u_S)
 -\tfrac12\langle\boldsymbol v_J,
 (C_{J\mid S}^{-1}\otimes I_d)\boldsymbol v_J\rangle.
\]
The quadratic factor has a rank-one Koszul contraction.  On the active
factor, write
\[
 \delta=\mathrm dH_{S,0}\wedge,\qquad
 \eta=\mathrm d(G_{S\mid S}-H_{S,0})\wedge.
\]
Tensoring the homogeneous contractions of
\cref{thm:jacobi-normal-form-transfer}, with the \(a\)-th homotopy
scaled by \((i\lambda_a)^{-1}\), gives a contraction
\((\iota,\pi,h)\) for \(\delta\) over \(A_S\).  Since
\[
 \deg(hf)\leq\deg f-(m-1),\qquad
 \deg(\eta f)\leq\deg f+1,
\]
\(\eta h\) lowers degree by \(m-2>0\).  Thus
\(U=\sum_{j\geq0}(-\eta h)^j\) is finite on every form.
The perturbation formulas give
\[
 \pi_G=\pi U,\qquad h_G=hU,\qquad
 (\delta+\eta)h_G+h_G(\delta+\eta)
 =\Id-\iota\pi_G.
\]
Here the inclusion is unchanged because \(\eta\iota=0\).
Tensor with the quadratic contraction.  The target is
\(V_S\otimes A_S[-N_K]\), of rank \(\mu^{|S|}\), and all identities
remain valid after specialising the invertible \(\lambda_a\).
\end{proof}

\begin{theorem}[Formal lattice and radial connection]
\label{thm:toric-corner-Rees}
Assume \eqref{eq:principal-isolated}.  The cohomology of
\((\Omega^\bullet_{A_S[\boldsymbol u]/A_S}[[\rho]],\mathfrak D_{K\mid S})\)
is concentrated in degree \(N_K\), and
\begin{equation}
 \mathcal L_{K\mid S}:=H^{N_K}(\mathfrak D_{K\mid S})
 \simeq\mathcal J(G_{K\mid S})[[\rho]]
 \simeq V_S\otimes_\C A_S[[\rho]].
 \label{eq:toric-Rees-lattice}
\end{equation}
These isomorphisms commute with specialisation at every
\(\boldsymbol\lambda\in(\C^\times)^S\).  Moreover,
\begin{equation}
 \rho^3\nabla_{\rho,K}^{\mathrm{phys}}\mathcal L_{K\mid S}
 \subset\mathcal L_{K\mid S},
 \label{eq:radial-connection-lattice}
\end{equation}
and the induced endomorphism modulo \(\rho\) is
\begin{equation}
 \rho^3\nabla_{\rho,K}^{\mathrm{phys}}\bmod\rho
 =-2m_{G_{K\mid S}}\quad\text{on }\mathcal J(G_{K\mid S}).
 \label{eq:leading-radial-endomorphism}
\end{equation}
Here \(m_G\) denotes multiplication by \(G\).
\end{theorem}

\begin{proof}
Let \((\iota_0,\pi_0,h_0)\) be the contraction in
\cref{lem:radial-Jacobian}, and set
\[
 \eta_\rho=\rho^2\mathrm d_{\boldsymbol u}
 +\sum_{j=1}^m\rho^j\mathrm dH_{S,j}\wedge.
\]
Since \(\eta_\rho\) is divisible by \(\rho\), the perturbation formulas
\cite{GugenheimLambe}
\[
 \pi_\rho=\pi_0\sum_{k\geq0}(-\eta_\rho h_0)^k,
 \qquad h_\rho=h_0\sum_{k\geq0}(-\eta_\rho h_0)^k
\]
are defined coefficientwise over \(A_S[[\rho]]\).
The inclusion remains \(\iota_0\), since \(\eta_\rho\iota_0=0\) on top
forms; the transferred differential is zero.  This proves
\eqref{eq:toric-Rees-lattice}, including specialisation.

On forms,
\begin{equation}
 \rho^3\nabla_{\rho,K}^{\mathrm{phys}}
 =\rho^3\partial_\rho+\rho\partial_\rho\Psi_{K\mid S}
 -2\Psi_{K\mid S}-N_K\rho^2,
 \label{eq:rescaled-radial-operator}
\end{equation}
and direct calculation gives
\begin{equation}
 [\rho^3\nabla_{\rho,K}^{\mathrm{phys}},\mathfrak D_{K\mid S}]
 =2\rho^2\mathfrak D_{K\mid S}.
 \label{eq:radial-Rees-commutator}
\end{equation}
Thus the operator preserves cycles and boundaries.  It induces
\eqref{eq:radial-connection-lattice}; reduction of
\eqref{eq:rescaled-radial-operator} modulo \(\rho\) gives
\eqref{eq:leading-radial-endomorphism}.
\end{proof}

The one-copy specialisation \(r=1\), \(\Sigma=(1)\), \(\lambda_1=1\)
is written out in \cref{thm:ramified-Brieskorn-lattice} in
Appendix~\ref{sec:one-variable-refinements}, together with the rescaled
phase and the resulting exponential levels.

\subsection{Formal periods}
\label{subsec:radial-formal-periods}

In this subsection, fix \(\boldsymbol\lambda\in(\C^\times)^S\), and
abbreviate \(G=G_{K\mid S}\), \(\Psi=\Psi_{K\mid S}\).
Write \(G=Q+i\Lambda\), where
\[
 Q(u)=-\tfrac12\langle u,(\Sigma_K^{-1}\otimes I_d)u\rangle,
 \qquad \Lambda(u)=\sum_{a\in S}\lambda_aP_m(u_a).
\]
For every critical point \(c\) of \(G\), Euler's identity gives
\[
 0=\langle c,\nabla G(c)\rangle=2Q(c)+im\Lambda(c),
 \qquad G(c)=\frac{m-2}{m}Q(c).
\]
In particular, a nonzero critical point can have value zero because the
complex quadratic form \(Q\) can vanish there.

The condition
\begin{equation}
 \operatorname{Crit}(G)\cap G^{-1}(0)=\{0\}
 \label{eq:central-nonresonance}
\end{equation}
will isolate the summand associated with the critical point at the
origin.  It does not require the nonzero critical values to be distinct.

\begin{proposition}[Decomposition by critical values]
\label{prop:radial-formal-blocks}
Assume \eqref{eq:principal-isolated} and that \(G\) is Morse.
The formal radial connection and its dual decompose canonically into
summands indexed by the critical values \(\gamma\) of \(G\).
The summand indexed by \(\gamma\) has rank
\[
 \#\{c\in\operatorname{Crit}(G):G(c)=\gamma\}.
\]
Under \eqref{eq:central-nonresonance}, the summand indexed by \(0\) has
rank one.  If all critical values are distinct, the formal period
corresponding to each critical point \(c\) has exponential factor
\begin{equation}
 \exp\!\left(\frac{G(c)}{\rho^2}
 +\frac{H_{S,1}(c)}{\rho}+O(1)\right).
 \label{eq:toric-formal-levels}
\end{equation}
\end{proposition}

\begin{proof}
Since \(G\) is Morse,
\[
 \mathcal J(G)=\bigoplus_\gamma E_\gamma,\qquad
 E_\gamma=\bigoplus_{G(c)=\gamma}\C e_c,\qquad
 m_G|_{E_\gamma}=\gamma\Id.
\]
In a basis lifting these spaces, write
\(\rho^3\nabla^{\mathrm{phys}}=\rho^3\partial_\rho+A(\rho)\), with
\(A(\rho)=\sum_{n\geq0}A_n\rho^n\) and
\(A_0|_{E_\gamma}=-2\gamma\Id\).
The spectra of the leading blocks are the distinct singletons
\(\{-2\gamma\}\).  The formal splitting lemma
\cite[Lemma~2.1]{PflugelSplitting} therefore gives
\(H=I+O(\rho)\) such that
\begin{equation}
 \rho^3H'+AH=HB,
 \label{eq:radial-block-gauge}
\end{equation}
with \(B\) block diagonal for \(\bigoplus_\gamma E_\gamma\).
Two invariant summands lifting \(E_\gamma\) cannot first differ by
\(\rho^nX:E_\gamma\to E_{\gamma'}\): invariance would give
\(2(\gamma-\gamma')X=0\).  Thus the summands are intrinsic.  Their ranks
are \(\dim E_\gamma\), and dualisation gives the period decomposition.

Let \(c(\rho)\) be the formal critical point of \(\Psi(\rho,\cdot)\)
with \(c(0)=c\).  The formal implicit-function theorem gives
\[
 \Psi(\rho,c(\rho))=G(c)+\rho H_{S,1}(c)+O(\rho^2).
\]
For \(u=c(\rho)+\rho v\),
\[
 \rho^{-2}\Psi(\rho,u)
 =\rho^{-2}\Psi(\rho,c(\rho))
 +\tfrac12v^{\mathsf T}Q(\rho)v
 +\sum_{k=3}^m\rho^{k-2}R_k(\rho,v),
\]
where \(Q(0)=\operatorname{Hess}G(c)\) is invertible.
Formal Gaussian integration yields
\(e^{\Psi(\rho,c(\rho))/\rho^2}a_c(\rho)\), with
\(a_c\in\C[[\rho]]\) and \(a_c(0)\neq0\): the volume factor
\(\rho^{N_K}\) cancels the normalisation \(\rho^{-N_K}\), and each
coefficient uses finitely many Gaussian moments.  Integration by parts
gives the dual connection equation.  The critical-value expansion gives
\eqref{eq:toric-formal-levels}.
\end{proof}

\begin{proposition}[Gaussian integration on the formal lattices]
\label{prop:conditional-Rees-face-maps}
For \(S\subset T\subset K\), define
\begin{equation}
 \mathsf W^\rho_{K\to T}f(\boldsymbol u_T)
 =\left.
 \exp\!\left(\frac{\rho^2}{2}
 \Delta_{C_{K\setminus T\mid T}}^{(\boldsymbol v)}\right)
 f(\boldsymbol u_T,L_{K\setminus T\mid T}\boldsymbol u_T+\boldsymbol v)
 \right|_{\boldsymbol v=0}.
 \label{eq:conditional-Rees-Wick}
\end{equation}
With the orientation convention of \cref{thm:gaussian-face-cube},
\(\mathsf W^\rho_{K\to T}\) induces an isomorphism
\[
 \mathbf W^{\rho,S}_{K\to T}:\mathcal L_{K\mid S}
 \xrightarrow{\ \simeq\ }\mathcal L_{T\mid S},\qquad
 [f\operatorname{vol}_K]\longmapsto
 [\mathsf W^\rho_{K\to T}f\operatorname{vol}_T].
\]
For \(S\subset T\subset K\subset L\),
\begin{equation}
 \mathsf W^\rho_{K\to T}\circ\mathsf W^\rho_{L\to K}
 =\mathsf W^\rho_{L\to T}.
 \label{eq:Rees-Wick-tower}
\end{equation}
The lattice maps satisfy the same composition law.
After inverting \(\rho\), they also satisfy
\begin{equation}
 \mathbf W^{\rho,S}_{K\to T}\nabla_{\rho,K}^{\mathrm{phys}}
 =\nabla_{\rho,T}^{\mathrm{phys}}\mathbf W^{\rho,S}_{K\to T}.
 \label{eq:physical-face-horizontality}
\end{equation}
\end{proposition}

\begin{proof}
Write \(n=d(|K|-|T|)\), \(C=C_{K\setminus T\mid T}\), and
\(L=L_{K\setminus T\mid T}\).  The Schur identity is
\begin{equation}
 \rho^{-2}\Psi_{K\mid S}(\rho,\boldsymbol u_T,L\boldsymbol u_T+\boldsymbol v)
 =\rho^{-2}\Psi_{T\mid S}(\rho,\boldsymbol u_T)
 -\frac{\langle\boldsymbol v,(C^{-1}\otimes I_d)\boldsymbol v\rangle}
 {2\rho^2}.
 \label{eq:conditional-Rees-phase-splitting}
\end{equation}
The quadratic contraction over \(A_S[[\rho]]\) has projection
\eqref{eq:conditional-Rees-Wick}; tensoring it with the active complex
gives the isomorphism.  Equation \eqref{eq:Rees-Wick-tower} is
\eqref{eq:conditional-Wick-tower} with covariance \(\rho^2\Sigma\).

For \(\rho>0\), Gaussian moments give
\begin{equation}
 \mathsf W^\rho_{K\to T}f
 =\frac{(2\pi)^{-n/2}\rho^{-n}}{(\det C)^{d/2}}
 \int_{\R^n}e^{-\langle v,(C^{-1}\otimes I_d)v\rangle/(2\rho^2)}
 f(\boldsymbol u_T,L\boldsymbol u_T+v)\,\dd v.
 \label{eq:conditional-Rees-integral}
\end{equation}
Differentiating the integral gives
\begin{equation}
 \partial_\rho\mathsf W^\rho_{K\to T}f
 =\mathsf W^\rho_{K\to T}\!\left(
 \partial_\rho f-\frac n\rho f
 +\frac{\langle v,(C^{-1}\otimes I_d)v\rangle}{\rho^3}f\right).
 \label{eq:conditional-Rees-radial-derivative}
\end{equation}
Together with \eqref{eq:conditional-Rees-phase-splitting} and
\(N_K=N_T+n\), this is \eqref{eq:physical-face-horizontality} on
representatives.  Its formal extension follows coefficientwise.
Multiplication by \(\rho^3\) gives an identity on the unlocalised
lattices by \eqref{eq:radial-connection-lattice}.
\end{proof}

The Hessian of \(G_{K\mid S}\) at \(0\) is
\(-\Sigma_K^{-1}\otimes I_d\).  Let \(u_{\mathrm G}(\rho,\boldsymbol\lambda)\)
be the unique formal critical point of \(\Psi_{K\mid S}\) satisfying
\(u_{\mathrm G}(0,\boldsymbol\lambda)=0\).  Denote by \(\widehat\Pi^0_{K,\rho}\) the formal
stationary-phase expansion at this point of
\begin{equation}
 (2\pi)^{-N_K/2}(\det\Sigma_K)^{-d/2}\rho^{-N_K}
 \int_{\Gamma_{\mathrm G}}e^{\rho^{-2}\Psi_{K\mid S}(\rho,u)}\,\dd u,
 \label{eq:central-formal-integral}
\end{equation}
with the local cycle oriented by the real Gaussian at \(\rho=0\).

For a polynomial amplitude \(f(\rho,u)\) with formal \(\rho\)-coefficients,
define
\begin{equation}
 \ell_{K,\rho}(f\operatorname{vol}_{u})
 =\left.
 e^{\Delta_{\Sigma_K\otimes I_d}/2}
 \left(f(\rho,\rho x)
 e^{i\rho^{m-2}\sum_{a\in K}\lambda_aP(x_a)}\right)
 \right|_{x=0}.
 \label{eq:central-Gaussian-functional}
\end{equation}
Each coefficient uses finitely many polynomial Gaussian moments.
For the amplitude \(1\), this gives
\begin{align}
 \widehat\Phi^{\mathrm G}_{P,K,\Sigma_K}(\rho,\boldsymbol\lambda)
 &:=\left.e^{\Delta_{\Sigma_K\otimes I_d}/2}
 e^{i\rho^{m-2}\sum_{a\in K}\lambda_aP(x_a)}\right|_{x=0}\notag\\
 &=\sum_{\alpha\in\N^K}
 \frac{i^{|\alpha|}\boldsymbol\lambda^\alpha}{\alpha!}
 \rho^{(m-2)|\alpha|}\E\prod_{a\in K}P(X_a)^{\alpha_a}.
 \label{eq:central-Gaussian-Wick-series}
\end{align}
For \(K=S=R\), before specialising \(\boldsymbol\lambda\), coefficient
extraction gives
\begin{equation}
 M_{P,r}(\Sigma)
 =i^{-r}[\rho^{(m-2)r}\lambda_1\cdots\lambda_r]
 \widehat\Phi^{\mathrm G}_{P,R,\Sigma}(\rho,\boldsymbol\lambda).
 \label{eq:normal-symbol-Torelli}
\end{equation}
This identity holds for every polynomial of degree at most \(m\).
For a fixed support \(S\), extend \(\ell_{K,\rho}\) continuously and
\(A_S[[\rho]]\)-linearly.

\begin{proposition}[The formal period at the origin]
\label{prop:marked-Gaussian-level}
For every polynomial \(P\) of degree at most \(m\),
\[
 \widehat\Pi^0_{K,\rho}
 =\ell_{K,\rho}(\operatorname{vol}_{u})
 =\widehat\Phi^{\mathrm G}_{P,K,\Sigma_K}
 (\rho,\boldsymbol\lambda).
\]
Under \eqref{eq:principal-isolated}, \(\ell_{K,\rho}\) induces a
horizontal functional on \(\mathcal L_{K\mid S}\), and, for
\(S\subset T\subset K\),
\[
 \ell_{K,\rho}
 =\ell_{T,\rho}\circ\mathbf W^{\rho,S}_{K\to T}.
\]
If moreover \(G_{K\mid S}\) is Morse and satisfies
\eqref{eq:central-nonresonance}, this functional generates the
horizontal solutions of the rank-one dual summand indexed by \(0\).
\end{proposition}

\begin{proof}
The formal implicit-function theorem uses the fixed inverse Hessian
\(-\Sigma_K\otimes I_d\); hence the coefficients of \(u_{\mathrm G}\)
are polynomial in \(\boldsymbol\lambda\).  The term independent of
\(u\) in the critical-point equation is
\(i\rho^{m-1}\sum_a\lambda_a\mathrm dP_1\).  Therefore
\(u_{\mathrm G}\in\rho^{m-1}\C[\boldsymbol\lambda][[\rho]]^{N_K}\), and
\begin{equation}
 \rho^{-2}\Psi_{K\mid S}(\rho,u_{\mathrm G})
 \in\rho^{m-2}\C[\boldsymbol\lambda][[\rho]].
 \label{eq:central-critical-action}
\end{equation}
In \eqref{eq:central-formal-integral}, the substitution \(u=\rho x\)
cancels \(\rho^{-N_K}\) and gives
\[
 \rho^{-2}\Psi_{K\mid S}(\rho,\rho x)
 =-\tfrac12\langle x,(\Sigma_K^{-1}\otimes I_d)x\rangle
 +i\rho^{m-2}\sum_a\lambda_aP(x_a).
\]
Expansion of the second exponential and Gaussian integration give
\eqref{eq:central-Gaussian-Wick-series}.  This is the local stationary-phase
expansion: the centre \(u_{\mathrm G}/\rho=O(\rho^{m-2})\) tends to
zero, and the Hessian of the rescaled phase at that centre tends to
the Gaussian Hessian.  To verify invariance
under the translation to that centre, put
\(B=\Sigma_K^{-1}\otimes I_d\) and
\(a=u_{\mathrm G}/\rho\).  For every polynomial \(g\),
\[
 \E\!\left[e^{-a^{\mathsf T}BX-a^{\mathsf T}Ba/2}g(X+a)\right]
 =\E g(X)
\]
as an identity in formal \(a\).  To prove it, use the formal generating
function \(g(x)=e^{z^{\mathsf T}x}\).  The left-hand side is
\[
 e^{z^{\mathsf T}a-a^{\mathsf T}Ba/2}
 e^{(z-Ba)^{\mathsf T}B^{-1}(z-Ba)/2}
 =e^{z^{\mathsf T}B^{-1}z/2};
\]
comparison of coefficients in \(z\) proves the polynomial identity.
Since \(a=O(\rho^{m-2})\), it applies coefficientwise also to the
perturbative exponential.  Translation to the critical centre removes
the linear term of the phase and gives the same formal Gaussian expansion.

For each coordinate \(u_j\), Gaussian integration by parts gives
\[
 \ell_{K,\rho}\bigl(\rho^2\partial_{u_j}f+
 (\partial_{u_j}\Psi_{K\mid S})f\bigr)=0.
\]
Thus \(\ell_{K,\rho}\) vanishes on boundaries.  Differentiating the normalised
formal integral, including \(\rho^{-N_K}\), gives
\begin{equation}
 \partial_\rho\ell_{K,\rho}([\omega])
 =\ell_{K,\rho}\bigl(\nabla_{\rho,K}^{\mathrm{phys}}[\omega]\bigr),
 \label{eq:dual-period-horizontality}
\end{equation}
which is dual horizontality.  Conditioning \(\rho X_K\) on \(\rho X_T\)
gives
\(\ell_{K,\rho}=\ell_{T,\rho}\circ\mathbf W^{\rho,S}_{K\to T}\).

Under the last hypotheses, apply \cref{prop:radial-formal-blocks}.
A nonzero component of \(\ell_{K,\rho}\) in a block \(\gamma\neq0\)
would have a first term \(\rho^kL\).  At order \(\rho^k\), equation
\eqref{eq:dual-period-horizontality}, multiplied by \(\rho^3\), would give
\(0=-2\gamma L\), a contradiction.  The remaining block has rank one,
and \(\ell_{K,\rho}([1])=1+O(\rho^{m-2})\).
\end{proof}

\begin{remark}[Real remainder estimate]
\label{cor:Gaussian-real-asymptotics}
For real \(P\) and \(\boldsymbol\lambda\) in a compact subset of
\(\R^K\), put \(Y_\lambda=\sum_a\lambda_aP(X_a)\).
Taylor's remainder and Gaussian moment bounds give, for \(n\geq1\),
\[
 \left|\E e^{itY_\lambda}
 -\sum_{j=0}^{n-1}\frac{(it)^j}{j!}\E Y_\lambda^j\right|
 \leq\frac{|t|^n}{n!}\E|Y_\lambda|^n
 \leq CA^n(n!)^{(m-2)/2}|t|^n.
\]
Thus \eqref{eq:central-Gaussian-Wick-series}, with
\(t=\rho^{m-2}\in\R\), is a real asymptotic expansion of Gevrey
order at most \((m-2)/2\).
\end{remark}

\begin{remark}[Vanishing parameters and equal critical values]
\label{rem:toric-corner-scope}
When an active \(\lambda_a\) vanishes, the rank changes from
\(\mu^{|S|}\) to \(\mu^{|S|-1}\); the lost critical points escape to
infinity.  The coefficients of \(u_{\mathrm G}\) and
\(\widehat\Phi^{\mathrm G}_{P,K,\Sigma_K}\) are polynomial in
\(\boldsymbol\lambda\), so both extend across these parameter values.

Condition \eqref{eq:central-nonresonance} is necessary for the zero-value
summand to have rank one.  For
\[
 d=2,\quad r=1,\quad m=6,\quad
 P_6(x,y)=\frac{x^6+y^6}{6},\quad\Sigma=(1),\quad\lambda_1=1,
\]
and \(a^4=-i\), the point \(c=(a,ia)\neq0\) is a Morse critical point
of \(G=-(x^2+y^2)/2+iP_6(x,y)\), with \(G(c)=G(0)=0\).
The local formal period at the origin remains defined, but its
critical-value summand contains both points.

Repeated critical values also occur for signed directions with \(\Sigma=I_r\).
For \(d=1\), \(P=x^4/4\), \(r=2\), and \(\lambda=(1,-1)\),
choose \(c_1^2=-i\) and \(c_2^2=i\).  Then
\[
 \nabla G(c_1,c_2)=0,\qquad
 \operatorname{Hess}G(c_1,c_2)=\operatorname{diag}(2,2),\qquad
 G(c_1,c_2)=-\tfrac14(c_1^2+c_2^2)=0.
\]
For \(P=x^5/5\), \(r=3\), and \(\lambda=(1,1,1)\), let
\(\zeta=e^{2\pi i/3}\) and take \(c=(i,i\zeta,i\zeta^2)\).
Each component satisfies \(c_a^3=-i\), so
\[
 \nabla G(c)=0,\qquad \operatorname{Hess}G(c)=3I_3,\qquad
 G(c)=-\tfrac3{10}\sum_{a=1}^3c_a^2=0.
\]
Thus the signed-direction formula for local formal periods below cannot
in general be replaced by a projection onto a rank-one zero-value block.
Also, if \(m\) is even, \(G(-u)=G(u)\); every nonzero critical point is
paired with a distinct point of the same value.

The substitution \(t_a=\rho^{m-2}\lambda_a\) treats fixed ratios.
A path such as \(t_1=\rho^{m-2}\), \(t_2=\rho^{2(m-2)}\) requires a
further rescaling to describe the escaping critical points.
\end{remark}

\section{Real-contour periods and Gaussian moments}
\label{sec:real-contour-periods}
\label{subsec:physical-betti}

We use the thimble construction and Picard--Lefschetz transformations of
\cite{Pham}, and the de Rham--rapid-decay pairing of
\cite[Theorem~5.2]{Hien}.  An admissible direction for a Morse phase \(F\)
is an angle \(\theta\in(-\pi/2,\pi/2)\) such that
\[
 \Im\bigl(e^{-i\theta}(F(c)-F(c'))\bigr)\neq0
 \quad\text{whenever }F(c)\neq F(c').
\]
Pairs with equal critical values impose no restriction.

\begin{lemma}[Thimbles with coinciding critical values]
\label{lem:equal-value-thimbles}
Let \(F\) be a Morse polynomial phase.  Assume that
\(\|\nabla F(z)\|\to\infty\) as \(\|z\|\to\infty\), and let
\(\theta\) be admissible.  The descending and ascending thimbles for
\(\Re(e^{-i\theta}F)\) form dual integral bases of the corresponding
relative homology groups.  Distinct critical points may have the same
critical value.
\end{lemma}

\begin{proof}
Write \(e^{-i\theta}F=f+ig\).  The bounded vector field
\(X=a\nabla f\), with \(a=(1+\|\nabla f\|^2)^{-1}\),
is complete.  Along its nonconstant trajectories,
\[
 \frac{\mathrm d f}{\mathrm ds}=a\|\nabla f\|^2>0,
 \qquad
 \frac{\mathrm d g}{\mathrm ds}
 =a\langle\nabla g,\nabla f\rangle=0.
\]
Thus a trajectory from a critical point \(p\) to a distinct critical
point \(q\) would satisfy
\[
 g(p)=g(q),\qquad f(p)<f(q).
\]
The first equality and admissibility force \(F(p)=F(q)\), contradicting
the second inequality.  There are therefore no trajectories between
distinct critical points, including those with equal critical values.

The gradient bound gives the relative Morse decomposition with one
middle-dimensional thimble per critical point \cite{Pham}.  If
\(\Gamma_p^-\cap\Gamma_q^+\neq\varnothing\), their common trajectory
connects \(p\) and \(q\); hence \(p=q\).  At \(p=q\), their tangent spaces are the complementary negative
and positive eigenspaces of \(\operatorname{Hess}f(p)\).
The intersection is transverse; choosing dual orientations gives
\[
 \langle\Gamma_p^-,\Gamma_q^+\rangle_{\mathrm{int}}
 =\delta_{pq}.
\]
Finally, on each descending end, \(g\) is constant and
\[
 \Re F=\cos\theta\,f-\sin\theta\,g\longrightarrow-\infty,
 \qquad \cos\theta>0.
\]
Thus the descending thimbles are rapid-decay cycles for \(e^F\).
\end{proof}

The gradient bound is automatic for the Gaussian phases with isolated
leading gradient and all couplings nonzero.  Their leading homogeneous
term \(F_m=i\sum_a\tau_aP_m(x_a)\) has gradient nonzero on the complex
unit sphere.  Compactness gives \(b>0\) such that
\[
 \|\nabla F(z)\|\geq b\|z\|^{m-1}
 -C(1+\|z\|^{m-2})\longrightarrow\infty.
\]
This also applies to the rescaled phase for every fixed \(\rho\neq0\).

Let \(\mathfrak C\) be a simply connected parameter domain on which the
phase is Morse with \(\mu\) critical points and satisfies the gradient
bound.  Choose a continuous admissible direction, and fix the
critical-point labels and orientations.  Assume that the coincidence
pattern of critical values is constant on this domain.  Write
\[
 H^-_\lambda=H_d(\C^d,\{f_\theta\ll0\};\mathbb Z),\qquad
 H^+_\lambda=H_d(\C^d,\{f_\theta\gg0\};\mathbb Z).
\]
The thimbles are transported horizontally inside \(\mathfrak C\), and
\begin{equation}
 \langle\Gamma_\alpha,\check\Gamma_\beta\rangle_{\mathrm{int}}
 =\delta_{\alpha\beta}.
 \label{eq:dual-thimble-normalisation}
\end{equation}
An \emph{admissible chamber} below means a domain with these phase,
direction, critical-point, and orientation data.

\begin{definition}[Coordinates in a thimble basis]
\label{def:physical-betti-matrix}
Let \(\gamma_1,\ldots,\gamma_N\) be horizontal integral rapid-decay
classes obtained by continuation of chosen oriented contours.  Their
cycle-coordinate matrix is
\begin{equation}
 C^{\mathrm{Betti}}_{u\alpha}
 :=
 \bigl\langle\gamma_u,\check\Gamma_\alpha\bigr\rangle_{\mathrm{int}}.
 \label{eq:physical-betti-matrix}
\end{equation}
\end{definition}

Scalar factors normalising Gaussian densities are placed in the
differential forms, not in the cycles.

The duality \eqref{eq:dual-thimble-normalisation} gives
\begin{equation}
 [\gamma_u]
 =\sum_{\alpha=1}^{\mu}C^{\mathrm{Betti}}_{u\alpha}[\Gamma_\alpha],
 \qquad
 C^{\mathrm{Betti}}_{u\alpha}
 =\#_{\mathrm{or}}(\gamma_u\pitchfork\check\Gamma_\alpha).
 \label{eq:physical-cycle-intersection-expansion}
\end{equation}
The intersection count uses transverse representatives with disjoint ends
at infinity.  It is finite, integral, and constant under horizontal
transport within the chamber.  The Smith normal form of
\(C_{\mathrm{Betti}}\) gives its integral kernel, cokernel, rank, and
elementary divisors.

\begin{proposition}[Change of thimble basis]
\label{thm:stokes-betti-reconstruction}
Use the same horizontally continued classes on both sides of a Stokes
wall.  If the integral descending thimble bases satisfy
\(\boldsymbol\Gamma^+=S_{\mathrm{top}}\boldsymbol\Gamma^-\), then
\begin{equation}
 C_{\mathrm{Betti}}^+
 =C_{\mathrm{Betti}}^-S_{\mathrm{top}}^{-1}.
 \label{eq:Betti-wall-crossing}
\end{equation}
Along a path from \(\lambda_0\) to \(\lambda_1\) crossing successive
walls with matrices \(S_1,\ldots,S_k\),
\begin{equation}
 C_{\mathrm{Betti}}(\lambda_1)
 =C_{\mathrm{Betti}}(\lambda_0)S_1^{-1}\cdots S_k^{-1}.
 \label{eq:finite-Betti-transport}
\end{equation}
\end{proposition}

\begin{proof}
Intersection duality gives
\(\check{\boldsymbol\Gamma}^{+}
 =(S_{\mathrm{top}}^{-1})^{\mathsf T}\check{\boldsymbol\Gamma}^{-}\).
Intersect with each horizontal class \(\gamma_u\), then iterate.
\end{proof}

For an analytic normalisation
\(\boldsymbol\Gamma^{\mathrm{an}}=G\boldsymbol\Gamma^{\mathrm{top}}\),
the same change-of-basis rule gives
\begin{equation}
 C_{\mathrm{sel}}=C_{\mathrm{Betti}}G^{-1}.
 \label{eq:analytic-physical-gauge}
\end{equation}
In particular, the multiplier \(i\) in the Airy formula
\eqref{eq:h3-exact-stokes-jump} refers to the analytic solution frame;
the coefficients in an oriented integral thimble frame are integers.

\begin{corollary}[Coordinates of the real Gaussian contour]
\label{cor:gaussian-replica-betti}
Let \(P\) be a real polynomial,
\(\Sigma\in\operatorname{Sym}_r^{++}(\R)\), and
\(\boldsymbol\tau\in(\R\setminus\{0\})^r\).  For
\[
 F_{\Sigma,\boldsymbol\tau}(\boldsymbol x)
 =
 -\frac12
 \sum_{a,b=1}^r
 (\Sigma^{-1})_{ab}\langle x_a,x_b\rangle
 +
 i\sum_{a=1}^r\tau_aP(x_a),
\]
the oriented real contour
\(\gamma_{\R}=\R^{rd}\) is a rapid-decay cycle.
On every admissible chamber, its integral thimble coordinates are
\begin{equation}
 C_\alpha^{\mathrm{Betti}}
 =
 \#_{\mathrm{or}}
 \bigl(\R^{rd}\pitchfork\check\Gamma_\alpha\bigr).
 \label{eq:real-gaussian-betti-row}
\end{equation}
At a diagonal covariance \(\Sigma=\operatorname{diag}(\sigma_1,\ldots,\sigma_r)\),
with product orientations,
\begin{equation}
 C_{\mathrm{Betti}}^{(r)}
 =
 C_{\mathrm{Betti},1}^{(1)}
 \otimes\cdots\otimes
 C_{\mathrm{Betti},r}^{(1)}.
 \label{eq:replica-Betti-Kunneth}
\end{equation}
Here \(C_{\mathrm{Betti},a}^{(1)}\) is the coordinate row for the
one-copy phase with variance \(\sigma_a\) and coupling \(\tau_a\).
Along a covariance path from this diagonal base point, contained in the
tame Morse locus and crossing finitely many Stokes walls, the row is
obtained by \eqref{eq:finite-Betti-transport}.
\end{corollary}

\begin{proof}
Rapid decay follows from the Gaussian estimate
\[
 \Re F_{\Sigma,\boldsymbol\tau}(\boldsymbol x)
 =
 -\frac12
 \bigl\langle\boldsymbol x,
  (\Sigma^{-1}\otimes I_d)\boldsymbol x\bigr\rangle
 \leq-c_\Sigma|\boldsymbol x|^2.
\]
For diagonal \(\Sigma\), the phase is a sum of one-copy phases.
The real contour, descending and ascending thimbles, and their
intersection pairing are external products.  Formula
\eqref{eq:physical-betti-matrix} therefore becomes the Kronecker product
\eqref{eq:replica-Betti-Kunneth}.  The last assertion is
\cref{thm:stokes-betti-reconstruction}.
\end{proof}

In one variable, the coordinates of an end-to-end contour are the signed
edges along a unique path in a spanning tree of decay sectors.  This gives an
explicit matrix with entries in \(\{0,\pm1\}\), and tensor products give the
corresponding split-phase formula.  The topological proof and the
monomial Coxeter-rank calculation are given in
Appendix~\ref{sec:one-dimensional-stokes-trees}; see
\cref{thm:one-dimensional-stokes-tree,prop:monomial-Coxeter-orbit,cor:split-Stokes-tree}.

Near zero coupling, the coordinates of a continued contour are locally
constant in each admissible chamber and change across a Stokes wall by
\eqref{eq:Betti-wall-crossing}.  The origin coefficient at a real starting
parameter, including one on a Stokes wall, is determined in
\cref{lem:lateral-origin-coefficient}.

\begin{lemma}[The origin coefficient at small real coupling]
\label{lem:lateral-origin-coefficient}
Let \(P\) be real, \(\Sigma>0\), and
\(\boldsymbol\lambda\in(\R^*)^r\).  Put
\[
 F_t(x)=-\tfrac12\langle x,(\Sigma^{-1}\otimes I_d)x\rangle
       +it\sum_a\lambda_aP(x_a),\qquad t>0.
\]
Assume that \(F_t\) is Morse and has the thimble bases of
\cref{lem:equal-value-thimbles} for all sufficiently small \(t\).
Orient the thimble at the critical branch \(x_0(t)=O(t)\) by the real
Gaussian orientation.  For each sufficiently small fixed \(t>0\), both
lateral continuations of the real contour as \(\theta\to0^\pm\) have
coefficient one on this thimble.  The thimble at any critical point
\(c\) with
\(\Re F_t(c)>0\) has coefficient zero in both bases.
\end{lemma}

\begin{proof}
The orientation-preserving change
\(x=(\Sigma^{1/2}\otimes I_d)y\) gives
\[
 F_t(y)=-\tfrac12\sum_jy_j^2+itH(y),\qquad H\in\R[y].
\]
Put \(f_{t,\theta}=\Re(e^{-i\theta}F_t)\), and denote the ascending
thimble at the Gaussian critical branch by \(\check\Gamma_{0,t,\theta}\).
The implicit-function theorem and the local Morse lemma give
\[
 y_0(t)=O(t),\qquad F_t(y_0(t))=itH(0)+O(t^2).
\]
At \(t=\theta=0\), the ascending manifold is \(i\R^{rd}\).
Its local transverse intersection with \(\R^{rd}\) therefore persists
with sign \(+1\).  Let \(L_{t,\theta}\) be a small local ascending
disk.  Positivity on its boundary, followed by monotonicity along outgoing
trajectories, gives a constant \(\kappa>0\) such that
\[
 \#_{\mathrm{or}}\bigl(\R^{rd}\pitchfork L_{t,\theta}\bigr)=1,
 \qquad
 f_{t,\theta}|_{\check\Gamma_{0,t,\theta}\setminus L_{t,\theta}}
 \geq\kappa
\]
for all sufficiently small \(t,\theta\).

Fix such a \(t>0\), and choose
\[
 M>1+\max_{c\in\operatorname{Crit}(F_t)}|\Re F_t(c)|,
 \qquad R=2\sqrt M,
 \qquad D_R=\{y\in\R^{rd}:|y|\leq R\}.
\]
Since \(f_{t,0}(y)=-|y|^2/2\), the disk \(D_R\) represents the real
contour in relative homology.  For admissible \(\theta\) sufficiently
close to zero, it represents its horizontal continuation and satisfies
\[
 \sup_{\partial D_R}f_{t,\theta}<-M,
 \qquad \sup_{D_R}f_{t,\theta}\longrightarrow0
 \quad(\theta\to0).
\]
Taking \(\sup_{D_R}f_{t,\theta}<\kappa\) excludes every nonlocal
intersection with \(\check\Gamma_{0,t,\theta}\), so its coefficient is
\(+1\).  If \(\Re F_t(c)>0\), further decrease \(|\theta|\) until
\[
 \sup_{D_R}f_{t,\theta}<f_{t,\theta}(c)
 \leq f_{t,\theta}|_{\check\Gamma_{c,t,\theta}}.
\]
Then \(D_R\cap\check\Gamma_{c,t,\theta}=\varnothing\), so that coefficient
is zero.  The argument applies to both signs of \(\theta\), with its
allowed size depending on \(t\).
\end{proof}

\subsection{Algebraic, analytic, and topological factors}

Let \(P\) satisfy \eqref{eq:principal-isolated}, and fix a tame Morse
chamber \(\mathfrak C\) for the phases
\(F_{P,t}=-q/2+itP\) as in \cref{subsec:physical-betti}, contained in the
regular locus of the rational frames used below.  Choose an
admissible direction, distinguished vanishing paths, and an oriented
integral thimble basis \(\Gamma_1,\ldots,\Gamma_\mu\), where
\(\mu=(m-1)^d\).  Let \(\omega_1,\ldots,\omega_\mu\) be any rational de Rham
frame; \cref{thm:jacobi-normal-form-transfer} supplies one such frame.

Let \(e_1,\ldots,e_\mu\) be the chosen Jacobian representatives.
Define \(B_{\mathrm{Jac}}\) by
\[
 [e_j\operatorname{vol}_x]
 =\sum_{\beta=1}^{\mu}[\omega_\beta](B_{\mathrm{Jac}})_{\beta j},
\]
and set
\begin{equation}
 \mathcal P_{\mathrm{th}}
 =\left(\int_{\Gamma_\alpha}e^{F_{P,t}}\omega_\beta
   \right)_{\alpha,\beta=1}^{\mu}.
 \label{eq:thimble-period-matrix}
\end{equation}
For the frame \(\omega_j=e_j\operatorname{vol}_x\), one has
\(B_{\mathrm{Jac}}=I_\mu\).  In an arbitrary rational frame,
\(B_{\mathrm{Jac}}\) is rational in the parameters by
\cref{cor:jacobi-comparison-determinant}.  The matrix
\(\mathcal P_{\mathrm{th}}\) contains analytic periods, including the Airy
functions of \cref{thm:h3-airy-stokes} in Appendix~\ref{sec:cubic}.
The full cubic period determinant is computed in
\cref{cor:H3-full-period-determinant}. The quartic real-contour period
has the Bessel representation of \cref{thm:quartic-Bessel-Stokes} in
Appendix~\ref{sec:quartic}.

\begin{proposition}[Period matrix and rank]
\label{thm:thimble-factorisation}
On \(\mathfrak C\), the matrix \(\mathcal P_{\mathrm{th}}\) is holomorphic
and invertible.  For selected rapid-decay cycles \(\gamma_1,\ldots,\gamma_N\),
their periods against the Jacobian representatives satisfy
\begin{equation}
 \Pi_{\mathrm{sel}}
 =C_{\mathrm{Betti}}\mathcal P_{\mathrm{th}}B_{\mathrm{Jac}},
 \qquad
 \operatorname{rank}\Pi_{\mathrm{sel}}
 =\operatorname{rank}C_{\mathrm{Betti}}.
 \label{eq:correct-global-factorisation}
\end{equation}
\end{proposition}

\begin{proof}
For the connection \(d+\mathrm dF_{P,t}\wedge\), the dual horizontal
section is \(e^{F_{P,t}}\).  The perfect de Rham--rapid-decay pairing
\cite[Theorem~5.2]{Hien} therefore has matrix
\(\mathcal P_{\mathrm{th}}\) in the chosen bases.  Both bases have
\(\mu\) elements by \cref{thm:jacobi-normal-form-transfer} and the
thimble theorem \cite{Pham}, so this matrix is invertible.  Its entries
are holomorphic by \cite[Theorem~3.5]{HienRoucairol}.  Expanding the cycles
by \eqref{eq:physical-cycle-intersection-expansion} and the forms by the
definition of \(B_{\mathrm{Jac}}\) gives the factorisation.
Both \(\mathcal P_{\mathrm{th}}\) and \(B_{\mathrm{Jac}}\) are invertible,
hence preserve rank.
\end{proof}

Across a Stokes wall,
\[
 \mathcal P_{\mathrm{th}}^+=S_{\mathrm{top}}\mathcal P_{\mathrm{th}}^-,
 \qquad
 C_{\mathrm{Betti}}^+=C_{\mathrm{Betti}}^-S_{\mathrm{top}}^{-1}.
\]
Thus the change of thimble basis leaves \(\Pi_{\mathrm{sel}}\) and its
rank unchanged.  Near coalescing critical points, these formulas apply
on each simply connected part of
the tame Morse locus.  Transport within that locus preserves the integral
thimble lattice.

For \(r\) correlated copies, use the tensor Jacobian basis of
\cref{cor:replicated-jacobi-transfer} and a rational de Rham frame in
\(rd\) variables.  With the same conventions,
\[
 \Pi_{\mathrm{Jac}}^{(r)}
 =\mathcal P_{\mathrm{th}}^{(r)}B_{\mathrm{Jac}}^{(r)},
 \qquad
 \Pi_{\mathrm{sel}}^{(r)}
 =C_{\mathrm{Betti}}^{(r)}
   \mathcal P_{\mathrm{th}}^{(r)}B_{\mathrm{Jac}}^{(r)}.
\]
The rational matrix \(B_{\mathrm{Jac}}^{(r)}\) and the thimble period
matrix are square of order \((m-1)^{rd}\).

The same formula applies after integrating out inactive variables.
For \(I\subset K\) and \(\alpha\in\N^{K\setminus I}\), the right side
of \eqref{eq:all-face-normal-jet}, on an admissible chamber, equals
\begin{equation}
 Z_{\Sigma_{II}}C_{\mathrm{Betti}}^I\mathcal P_{\mathrm{th}}^I
 B_{\mathrm{Jac}}^I
 v_I\!\left(\mathsf W^{\Sigma_K}_{K\to I}
       \prod_{b\in K\setminus I}P(x_b)^{\alpha_b}\right),
 \label{eq:face-normal-jet-factorisation}
\end{equation}
where \(Z_{\Sigma_{II}}=(2\pi)^{-d|I|/2}(\det\Sigma_{II})^{-d/2}\)
and \(v_I(f)\) denotes the Jacobian coordinates of
\([f\operatorname{vol}_I]\).  For \(I=\varnothing\), the three matrices
are \(1\).

The full comparison from the Jacobian frame to thimble periods is
\(B_{\mathrm{global}}=\mathcal P_{\mathrm{th}}B_{\mathrm{Jac}}\).
The cycle matrix \(C_{\mathrm{Betti}}\) records intersections, while the
matrix \(A_r\) in \eqref{eq:finite-probe-gram} evaluates contractions at
covariances.  Their relation is through moment derivatives, as follows.
The cyclicity criterion for one contour is stated in
Appendix~\ref{sec:selected-contour-cyclicity},
\eqref{eq:period-cyclicity}; its quartic parity example uses
\cref{prop:quartic-parity-rank} in Appendix~\ref{sec:quartic}.
For split monomial phases with independent parameters,
\cref{cor:diagonal-physical-rank} in
Appendix~\ref{sec:one-variable-refinements} computes the rank of the
real-contour moment jets as a product of one-variable ranks.

\subsection{Recovery of the moment vector}

For \(P\in\mathcal P_{d,m}\), \(r\geq1\), and
\(\Sigma\in\operatorname{Sym}_r^{++}(\R)\), write
\[
 \Phi_{P,r,\Sigma}(\boldsymbol\tau)
 =\E\exp\!\left(i\sum_{a=1}^r\tau_aP(X_a^\Sigma)\right),
 \qquad \boldsymbol\tau\in\R^r.
\]

\begin{corollary}[Moment derivatives]
\label{thm:physical-period-torelli}
The function \(\Phi_{P,r,\Sigma}\) is \(C^\infty\), and
\begin{equation}
 i^{-r}\partial_{\tau_1}\cdots\partial_{\tau_r}
 \Phi_{P,r,\Sigma}(0)=M_{P,r}(\Sigma).
 \label{eq:physical-mixed-jet}
\end{equation}
Consequently, these derivatives for \(r\leq r_*(d,m)\) and
\(\Sigma\in\mathcal S_r\) determine the orthogonal orbit of \(P\).
\end{corollary}

\begin{proof}
Every derivative of the integrand is bounded by a polynomial in the
Gaussian vector, which is integrable.  Differentiation under the integral
gives \eqref{eq:physical-mixed-jet}; apply \cref{thm:intro-torelli}.
\end{proof}

For \(m\geq3\), \eqref{eq:normal-symbol-Torelli} expresses these moments
as coefficients of the Gaussian series.  We next identify this series
with the origin projection of a continued real-contour period under
the geometric hypotheses.

\subsection{The Gaussian period in the formal decomposition}

Put \(K=\{1,\ldots,r\}\), assume \(m\geq3\) and \eqref{eq:principal-isolated}.  For
\(\boldsymbol\lambda\in(\C^*)^r\), let the leading rescaled phase
\[
 G(\boldsymbol u)
 =-\frac12\langle\boldsymbol u,
       (\Sigma^{-1}\otimes I_d)\boldsymbol u\rangle
   +i\sum_{a=1}^r\lambda_aP_m(u_a)
\]
be Morse and satisfy \eqref{eq:central-nonresonance}.
By \cref{prop:radial-formal-blocks}, the critical branch issuing from
\(\boldsymbol u=0\) determines a rank-one summand
\(\mathcal E_0^\vee\) of the localised formal dual period module
\(\mathcal P_K^\vee\), where, after specialising the fixed couplings,
\[
 \mathcal P_K^\vee=
 \operatorname{Hom}_{\C((\rho))}
 \bigl(\mathcal L_{K\mid K}\otimes\C((\rho)),\C((\rho))\bigr)
\]
equipped with the dual physical radial connection.  Denote the projection
onto \(\mathcal E_0^\vee\) by
\[
 \pi_0^\vee:\mathcal P_K^\vee\longrightarrow\mathcal E_0^\vee.
\]
For stationary-phase expansions, extend this projection by scalars to
the differential algebra generated over \(\C((\rho))\) by the formal
exponential factors \(\exp(\gamma/\rho^2+b/\rho)\) that occur in those
expansions.  This defines its action also on formal period covectors
with nonzero exponential factors.
The horizontal Gaussian functional \(\ell_{K,\rho}\) of
\cref{prop:marked-Gaussian-level} satisfies
\begin{equation}
 \ell_{K,\rho}([1])
 =\widehat\Phi^{\mathrm G}_{P,r,\Sigma}
 =1+O(\rho^{m-2}).
 \label{eq:central-dual-normalisation}
\end{equation}

\begin{definition}[Projection of a formal period]
\label{def:marked-central-regular-part}
For a formal period covector \(\widehat\Lambda_\rho\) and a cohomology
class \(v\), both expressed in the rescaled \(u\)-coordinates, set
\begin{equation}
 \operatorname{Reg}^{\mathrm G}_0(\widehat\Lambda_\rho;v)
 =(\pi_0^\vee\widehat\Lambda_\rho)(v).
 \label{eq:marked-central-projection}
\end{equation}
For a sectorial period, the hat denotes its formal expansion in the
exponential decomposition on that sector.
\end{definition}

We now compare this projection with integration over a continued real
contour.  Put \(\boldsymbol\tau=\rho^{m-2}\boldsymbol\lambda\) and
\(Z_\Sigma=(2\pi)^{-rd/2}(\det\Sigma)^{-d/2}\).  In the
\(x\)-coordinate Jacobian frame, its period covector is
\[
 \Lambda_\rho
 =Z_\Sigma C_{\mathrm{Betti}}^{(r)}
   \mathcal P_{\mathrm{th}}^{(r)}B_{\mathrm{Jac}}^{(r)}.
\]
To express it in the rescaled frame, use the cohomological isomorphism
\begin{equation}
 T_\rho:
 [f(\boldsymbol u)\operatorname{vol}_{\boldsymbol u}]
 \longmapsto
 [f(\rho\boldsymbol x)\operatorname{vol}_{\boldsymbol x}],
 \qquad
 \widetilde\Lambda_\rho=\Lambda_\rho\circ T_\rho.
 \label{eq:radial-period-frame}
\end{equation}
The map \(T_\rho\) is \(\rho^{-rd}\) times pullback by
\(\boldsymbol u=\rho\boldsymbol x\); in particular,
\(T_\rho[1]_u=[1]_x\).  If \(\gamma_\rho\) is the continued contour,
then
\begin{equation}
 \widetilde\Lambda_\rho
 ([f(\boldsymbol u)\operatorname{vol}_{\boldsymbol u}])
 =Z_\Sigma\rho^{-rd}\int_{\rho\gamma_\rho}
 f(\boldsymbol u)e^{\rho^{-2}\Psi_{K\mid K}(\rho,\boldsymbol u)}
 \,\operatorname{vol}_{\boldsymbol u}.
 \label{eq:rescaled-period-integral}
\end{equation}

\begin{theorem}[Projection onto the origin period]
\label{thm:normal-symbol-torelli}
Let \(P\in\R[x_1,\ldots,x_d]\) have degree \(m\geq3\), let
\(\Sigma>0\), and let \(\boldsymbol\lambda\in(\C^*)^r\).
Assume \eqref{eq:principal-isolated}, that \(G\) is Morse, and
\(\operatorname{Crit}(G)\cap G^{-1}(0)=\{0\}\).
On a sector \(\Omega\subset\{0<|\rho|<\rho_0\}\) with
the admissible chamber data of \cref{subsec:physical-betti}, let \(\Gamma_{0,\rho}\) be the origin thimble
in the \(x\)-coordinates, with Gaussian orientation.  If the continued real contour satisfies
\[
 \langle\gamma_\rho,\check\Gamma_{0,\rho}\rangle_{\mathrm{int}}=1,
\]
then its formal period in the rescaled frame satisfies
\begin{equation}
 \pi_0^\vee\widehat{\widetilde\Lambda}_\rho=\ell_{K,\rho}.
 \label{eq:central-period-identification}
\end{equation}
\end{theorem}

\begin{proof}
Let \(\Gamma_{0,\rho}\) be the origin thimble and
\(\widetilde\Lambda_{0,\rho}\) its normalised period.  After
\(\boldsymbol u=\rho\boldsymbol x\), its formal stationary-phase
expansion \cite{Pham} is
\begin{align*}
 \widehat{\widetilde\Lambda}_{0,\rho}
 ([f(\boldsymbol u)\operatorname{vol}_{\boldsymbol u}])
 &=\left.
 e^{\Delta_{\Sigma\otimes I_d}/2}
 \left(f(\rho\boldsymbol x)
 e^{i\rho^{m-2}\sum_a\lambda_aP(x_a)}\right)
 \right|_{\boldsymbol x=0}\\
 &=\ell_{K,\rho}([f\operatorname{vol}_{\boldsymbol u}]).
\end{align*}
The second equality is \eqref{eq:central-Gaussian-functional}.
By \eqref{eq:central-nonresonance}, the formal period of every other
thimble lies in a nonzero irregular summand.  Projection removes those
terms and retains
\(\widehat{\widetilde\Lambda}_{0,\rho}\) with coefficient \(1\).
\end{proof}

For real \(\boldsymbol\lambda\in(\R^*)^r\),
\cref{lem:lateral-origin-coefficient}, with \(t=\rho^{m-2}\), gives
\[
 \bigl\langle\gamma_{\R}^{\pm},
       \check\Gamma_{0,\rho}^{\pm}\bigr\rangle_{\mathrm{int}}=1
 \qquad (0<\rho\ll1),
\]
where \(\pm\) denotes angular continuation as \(\theta\to0^\pm\)
at fixed real \(\rho\).  This supplies the coefficient-one hypothesis
at a real starting parameter.  Transport within an admissible chamber
preserves it; crossing a Stokes wall uses \eqref{eq:Betti-wall-crossing}.

\begin{corollary}[Moment coefficients]
\label{cor:projected-period-moments}
Assume the hypotheses of \cref{thm:normal-symbol-torelli}.
Let \(e_{[1]}\) represent \([\operatorname{vol}_{\boldsymbol u}]\)
in the rescaled Jacobian frame.  Take coefficients in the polynomial
extension \(\widehat\Phi^{\mathrm G}_{P,r,\Sigma}
\in\C[\boldsymbol\lambda][[\rho]]\) of the projected period.  Then
\begin{equation}
 M_{P,r}(\Sigma)
 =i^{-r}[\rho^{(m-2)r}\lambda_1\cdots\lambda_r]
   \operatorname{Reg}^{\mathrm G}_0
   (\widehat{\widetilde\Lambda}_\rho;e_{[1]}).
 \label{eq:factorised-normal-symbol-Torelli}
\end{equation}
\end{corollary}

\begin{proof}
Evaluate \eqref{eq:central-period-identification} at \(e_{[1]}\) and
apply \eqref{eq:normal-symbol-Torelli}.
\end{proof}

One may recover the mixed derivative from finitely many radial directions.
For \(\boldsymbol\varepsilon\in\{-1,1\}^r\), put
\(\phi_{\boldsymbol\varepsilon}(t)
 =\Phi_{P,r,\Sigma}(t\boldsymbol\varepsilon)\) and
\(D_{\boldsymbol\varepsilon}=\sum_a\varepsilon_a\partial_{\tau_a}\).
The multinomial formula gives
\[
 \sum_{\boldsymbol\varepsilon}
 \left(\prod_a\varepsilon_a\right)D_{\boldsymbol\varepsilon}^r
 =2^rr!\,\partial_{\tau_1}\cdots\partial_{\tau_r}.
\]
Together with \eqref{eq:physical-mixed-jet}, this yields
\begin{equation}
 M_{P,r}(\Sigma)
 =\frac{i^{-r}}{2^rr!}
   \sum_{\boldsymbol\varepsilon\in\{-1,1\}^r}
   \left(\prod_a\varepsilon_a\right)
   \phi_{\boldsymbol\varepsilon}^{(r)}(0).
 \label{eq:radial-polarisation}
\end{equation}
Let \(\widehat\Pi_{0,\boldsymbol\varepsilon}\) denote the normalised
local formal period at the origin for
\(\boldsymbol\lambda=\boldsymbol\varepsilon\).  By
\cref{thm:intro-period-realisation},
\begin{equation}
 M_{P,r}(\Sigma)
 =\frac{i^{-r}}{2^r}
   \sum_{\boldsymbol\varepsilon\in\{-1,1\}^r}
   \left(\prod_a\varepsilon_a\right)
   [\rho^{(m-2)r}]\widehat\Pi_{0,\boldsymbol\varepsilon}.
 \label{eq:radial-factorised-symbol}
\end{equation}
Indeed, the indicated coefficient is
\(i^r\E(\sum_a\varepsilon_aP(X_a))^r/r!\).  For \(|\alpha|=r\),
\[
 \sum_{\boldsymbol\varepsilon\in\{-1,1\}^r}
 \left(\prod_a\varepsilon_a\right)\boldsymbol\varepsilon^\alpha
 =\prod_{a=1}^r\left(\sum_{\varepsilon_a=\pm1}
                          \varepsilon_a^{\alpha_a+1}\right).
\]
This vanishes unless every \(\alpha_a\) is odd.  The condition
\(|\alpha|=r\) then forces every \(\alpha_a=1\), and the sum is \(2^r\).
Formula \eqref{eq:radial-factorised-symbol} holds for every real
polynomial of degree at most \(m\).  It requires neither distinct
critical values nor \eqref{eq:central-nonresonance}.

For \(r\leq r_*(d,m)\) and \(\Sigma\in\mathcal S_r\), these coefficients
recover exactly \(\mathbf J_{d,m}(P)\), and hence the orthogonal orbit
by \cref{thm:intro-torelli}.  The formulas concern formal coefficient
extraction, rather than finitely many values of a characteristic
function at nonzero coupling.  They use no moment-determinacy assumption.

If a given direction also satisfies the hypotheses of
\cref{thm:normal-symbol-torelli}, its local period in the preceding
formula can be obtained by the projection
\eqref{eq:central-period-identification}.  When other critical points have
leading value zero, the block with leading value zero can have rank
greater than one.  The local period attached to the specified origin remains defined, but the whole
block cannot be identified with that period.  The examples in
\cref{rem:toric-corner-scope} show why this distinction is required even
at covariances and signed directions in the finite reconstruction.

\section{Further questions}
\label{sec:scope-conclusion}

The minimal number of copies and the minimal covariance grid are unknown.
The bound \(r_*(d,m)\) uses a general invariant-degree estimate.  An
algorithm that recovers an orbit representative with a proved complexity
bound remains to be constructed.

Distinct Morse critical points can have the same critical value.  The
formal blocks then have higher rank; the origin summand remains defined
under \eqref{eq:central-nonresonance}.  Extracting the origin period from the entire block without specifying
the local critical point requires further information.  The local
stationary-phase formula itself remains valid at these collisions.

Couplings that vanish at different rates require different rescalings.
Constructing the resulting lattices and their comparison maps is a
separate problem from the fixed-support radial case treated here.

\appendix

\section{The cubic integral and its Airy representation}
\label{sec:cubic}

Let
\[
 H_3(x)=x^3-3x,
 \qquad
 \varphi_3(t)=\frac1{\sqrt{2\pi}}
 \int_{\R}e^{-x^2/2+itH_3(x)}\dd x.
\]
On a simply connected sector at \(t=0\) containing the positive real
axis, choose \(t^{1/3}>0\) for \(t>0\), and set
\begin{align}
 z(t)&=\frac{1-36t^2}{12\,3^{1/3}t^{4/3}},
 \label{eq:h3-airy-z}\\
 C(t)&=\sqrt{2\pi}(3t)^{-1/3}
       e^{1/(108t^2)-1/2},
 \label{eq:h3-airy-C}\\
 \zeta(t)&=\frac23z(t)^{3/2}
       =\frac{(1-36t^2)^{3/2}}{108t^2}.
 \label{eq:h3-airy-zeta}
\end{align}
With \(\omega=e^{2\pi i/3}\), define
\begin{align}
 F_0(t)&=C(t)\Ai(z(t)),\notag\\
 F_+(t)&=C(t)e^{i\pi/6}\Ai(\omega z(t)),
 \label{eq:h3-Fplus}\\
 F_-(t)&=C(t)e^{-i\pi/6}\Ai(\omega^2z(t)).
 \label{eq:h3-Fminus}
\end{align}

\begin{theorem}[Airy representation and Stokes multiplier]
\label{thm:h3-airy-stokes}
For \(t>0\),
\begin{equation}
 \varphi_3(t)=F_0(t).
 \label{eq:h3-exact-Ai}
\end{equation}
The functions \(F_0,F_+,F_-\) satisfy
\begin{equation}
 27t^3(18t^2+1)y''
 +(486t^4+99t^2+1)y'
 +6t(18t^2+1)^2y=0,
 \label{eq:h3-airy-ode}
\end{equation}
the pullback of \(w''=zw\) by \(z(t)\), with gauge \(C(t)\).
The lateral solutions \(F_+,F_-\) have the same dominant formal
expansion along the positive real Stokes ray and satisfy
\begin{equation}
 F_+-F_-=iF_0.
 \label{eq:h3-exact-stokes-jump}
\end{equation}
Thus the ordered pairs \(\mathcal Y_\pm=(F_0,F_\pm)\) satisfy
\begin{equation}
 \mathcal Y_+=\mathcal Y_-
 \begin{pmatrix}1&i\\0&1\end{pmatrix}.
 \label{eq:h3-stokes-matrix}
\end{equation}
\end{theorem}

\begin{proof}
Set
\[
 I_n(t)=\E\!\left[X^ne^{itH_3(X)}\right].
\]
Gaussian integration by parts gives
\[
 I_{n+1}
 =
 nI_{n-1}+3it(I_{n+2}-I_n),
 \qquad n\geq0.
\]
For \(n=0,1,2\), this yields
\[
 I_2=I_0+\frac{I_1}{3it},
 \qquad
 I_3=\left(1-\frac1{9t^2}\right)I_1,
 \qquad
 I_4=I_0+\frac{i}{27t^3}I_1.
\]
Since
\[
 I_n'=i(I_{n+3}-3I_{n+1}),
\]
the functions \(y_0=I_0\) and \(y_1=I_1\) satisfy
\[
 y_0'
 =
 -i\left(2+\frac1{9t^2}\right)y_1,
\]
and
\[
 y_1'
 =
 -2iy_0-
 \left(\frac1t+\frac1{27t^3}\right)y_1.
\]
Eliminating \(y_1\) gives \eqref{eq:h3-airy-ode}.

For \(t>0\), set
\[
 x=u-\frac{i}{6t}.
\]
Then
\begin{equation}
 -\frac{x^2}{2}+it(x^3-3x)
 =
 itu^3
 +i\left(\frac1{12t}-3t\right)u
 +\frac1{108t^2}-\frac12.
\label{eq:h3-completed-cubic}
\end{equation}
Put \(v=(3t)^{1/3}u\).  Then
\[
 itu^3+
 i\left(\frac1{12t}-3t\right)u
 =
 i\left(\frac{v^3}{3}+z(t)v\right).
\]
The image of the real \(x\)-axis is the horizontal line
\[
 \R+\frac{i(3t)^{1/3}}{6t}
\]
in the \(v\)-plane.  This line may be moved to the real axis.  Indeed, on
either vertical side \(v=\pm R+iy\), with
\(0\leq y\leq h=(3t)^{1/3}/(6t)\), one has
\[
 \Re\!\left[
 i\left(\frac{v^3}{3}+z(t)v\right)
 \right]
 =
 -R^2y+\frac{y^3}{3}-z(t)y.
\]
The two vertical integrals are therefore \(O(R^{-2})\).  Letting
\(R\to\infty\) and applying \cite[Eq.~(9.5.1)]{DLMF} gives
\[
 \begin{split}
 \varphi_3(t)
 &=
 \frac{e^{1/(108t^2)-1/2}}
 {\sqrt{2\pi}(3t)^{1/3}}
 \int_{\R}e^{i(v^3/3+z(t)v)}\dd v\\
 &=
 \sqrt{2\pi}(3t)^{-1/3}
 e^{1/(108t^2)-1/2}\Ai(z(t)),
 \end{split}
\]
which is \eqref{eq:h3-exact-Ai}.  Direct differentiation using
\(\Ai''(z)=z\Ai(z)\) recovers \eqref{eq:h3-airy-ode}.

The Airy connection formula \cite[Eq.~(9.2.12)]{DLMF} is
\begin{equation}
 \Ai(z)+\omega\Ai(\omega z)+\omega^2\Ai(\omega^2z)=0.
\label{eq:airy-rotation-relation}
\end{equation}
It gives
\[
 e^{i\pi/6}\Ai(\omega z)
 -
 e^{-i\pi/6}\Ai(\omega^2z)
 =
 i\Ai(z).
\]
Multiplication by \(C(t)\) proves
\eqref{eq:h3-exact-stokes-jump}.

For \(\zeta=\frac23z^{3/2}\), the Airy asymptotic expansion
\cite[Eq.~(9.7.5)]{DLMF} gives
\[
 \Ai(z)
 \sim
 \frac{z^{-1/4}}{2\sqrt\pi}
 e^{-\zeta}\widehat a_-(\zeta^{-1}),
\]
whereas, on the two lateral branches,
\[
 e^{i\pi/6}\Ai(\omega z)
 \sim
 e^{-i\pi/6}\Ai(\omega^2z)
 \sim
 \frac{z^{-1/4}}{2\sqrt\pi}
 e^{+\zeta}\widehat a_+(\zeta^{-1}),
\]
with
\(\widehat a_\pm\in1+\zeta^{-1}\C[[\zeta^{-1}]]\).
Thus \(F_+\) and \(F_-\) have the same formal expansion and differ by
\(iF_0\), proving
\eqref{eq:h3-stokes-matrix}.
\end{proof}

\begin{corollary}[Determinant of the cubic period matrix]
\label{cor:H3-full-period-determinant}
Put
\[
 a(t)=2+\frac1{9t^2},
 \qquad
 G_0(t)=\frac{iF_0'(t)}{a(t)},
 \qquad
 G_-(t)=\frac{iF_-'(t)}{a(t)}.
\]
Then \(G_0\) and \(G_-\) are the periods of the Jacobi class \([x]\) on
the cycles represented by \(F_0\) and \(F_-\), respectively.  In the
ordered cycle basis \((\Gamma_0,\Gamma_-)\) and cohomology basis
\(([1],[x])\), the full period matrix is
\[
 Y_{H_3}(t)
 =
 \begin{pmatrix}
  F_0(t)&G_0(t)\\
  F_-(t)&G_-(t)
 \end{pmatrix}
\]
and
\begin{equation}
 \det Y_{H_3}(t)
 =
 -\frac{i}{3t}
 \exp\!\left(\frac1{54t^2}-1\right).
 \label{eq:H3-full-period-determinant}
\end{equation}
The determinant is nonzero for every \(t\neq0\) in the chosen sector,
including points where \(1-36t^2=0\).
\end{corollary}

\begin{proof}
The first Gauss--Manin equation in the proof of
\cref{thm:h3-airy-stokes} is
\[
 F'=-ia(t)G,
\]
which gives the formulas for \(G_0,G_-\).  With
\(W(f,g)=fg'-f'g\),
\[
 \det Y_{H_3}
 =
 \frac{i}{a(t)}W_t(F_0,F_-).
\]
The Airy Wronskian in our normalisation is
\[
 W_z\!\left(
  \Ai(z),e^{-i\pi/6}\Ai(\omega^2z)
 \right)
 =
 \frac1{2\pi}.
\]
By \cite[Eqs.~(9.2.3)--(9.2.4)]{DLMF},
\[
 \Ai(0)\Ai'(0)=-\frac1{2\pi\sqrt3},
 \qquad
 e^{-i\pi/6}(\omega^2-1)=-\sqrt3.
\]
The common gauge \(C(t)\) therefore gives
\[
 W_t(F_0,F_-)=\frac{C(t)^2z'(t)}{2\pi}.
\]
Direct differentiation of \eqref{eq:h3-airy-z} yields
\[
 \frac{z'(t)}{a(t)}=-(3t)^{-1/3},
\]
while
\[
 C(t)^2
 =
 2\pi(3t)^{-2/3}
 \exp\!\left(\frac1{54t^2}-1\right).
\]
Substitution proves \eqref{eq:H3-full-period-determinant}.
\end{proof}

\begin{remark}[Normalisation of the formal solutions]
\label{rem:h3-formal-scales}
One has
\[
 \frac{C(t)}{2\sqrt\pi}z(t)^{-1/4}
 =
 (1-36t^2)^{-1/4}
 \exp\!\left(\frac1{108t^2}-\frac12\right),
\]
and
\begin{align*}
 \frac1{108t^2}-\frac12-\zeta(t)
 &=-\frac92t^2-27t^4+O(t^6),\\
 \frac1{108t^2}-\frac12+\zeta(t)
 &=\frac1{54t^2}-1+\frac92t^2+27t^4+O(t^6).
\end{align*}
Hence the two formal exponential parts are \(0\) and \(1/(54t^2)\).
In the Airy normalisation above, the dominant solution begins with
\[
 e^{1/(54t^2)-1}(1+O(t^2))
\]
and the multiplier is \(i\).  If the dominant formal solution is instead
normalised to start with \(e^{1/(54t^2)}\), then the multiplier becomes
\(ie\).  Reversing the orientation of the subdominant cycle changes the sign of the Stokes multiplier.
\end{remark}

\begin{remark}[Stokes and equal-magnitude rays]
\label{rem:h3-rays}
The leading exponential difference is
\[
 \Delta(t)=\frac1{54t^2}.
\]
Thus
\[
 \Re\Delta(t)=0
 \quad\Longleftrightarrow\quad
 \arg t=\frac\pi4+\frac{k\pi}{2};
\]
these are the equal-magnitude, or anti-Stokes, rays in the standard
exact-WKB convention.  The rays carrying the triangular Stokes matrices
satisfy
\[
 \Im\Delta(t)=0
 \quad\Longleftrightarrow\quad
 \arg t=\frac{k\pi}{2}.
\]
We use this convention for Stokes and anti-Stokes rays throughout.
\end{remark}

\section{The quartic integral and parity}
\label{sec:quartic}

Let \(Z\sim\mathcal N(0,1)\) and set
\[
 \Phi_4(t)=\E e^{itZ^4/4},\qquad
 z=\frac{i}{8t},\qquad
 A(t)=\frac1{2\sqrt\pi}\sqrt{\frac it}\,e^z,
\]
where the square root is positive for \(t=i\tau\), \(\tau>0\).
Set \(\nu=1/4\). On the ray \(z>0\), define
\[
 \Phi_{\exp}^{+}=A(t)iK_\nu(e^{i\pi}z),
 \qquad
 \Phi_{\exp}^{-}=-A(t)iK_\nu(e^{-i\pi}z).
\]

\begin{theorem}[Bessel representation and Stokes multiplier]
\label{thm:quartic-Bessel-Stokes}
For \(t=i\tau\) with \(\tau>0\),
\begin{equation}
 \Phi_4(t)=A(t)K_{1/4}(z).
 \label{eq:quartic-Bessel-uniformisation}
\end{equation}
This identity extends by analytic continuation.
The function \(\Phi_4\) satisfies
\begin{equation}
 16t^2\Phi_4''+(32t+4i)\Phi_4'+3\Phi_4=0.
 \label{eq:quartic-ODE}
\end{equation}
Equation~\eqref{eq:quartic-ODE} has formal exponential parts
\(0\) and \(i/(4t)\). The lateral exponential solutions satisfy
\begin{equation}
 \Phi_{\exp}^{+}-\Phi_{\exp}^{-}=i\sqrt2\,\Phi_4.
 \label{eq:quartic-Bessel-Stokes-jump}
\end{equation}
The Stokes rays are \(\arg t=\pi/2\pmod\pi\); the equal-magnitude
rays are \(\arg t=0\pmod\pi\).
\end{theorem}

\begin{proof}
For \(a,b>0\), set \(x=\sqrt{b/a}\sinh(v/4)\) and \(z=b^2/(8a)\).
Then
\[
 ax^4+bx^2=z(\cosh v-1),
 \qquad
 \dd x=\frac14\sqrt{\frac ba}\cosh(v/4)\dd v.
\]
The representation \cite[Eq.~(10.32.9)]{DLMF} therefore gives
\[
 \int_0^\infty e^{-ax^4-bx^2}\dd x
 =
 \frac14\sqrt{\frac ba}\,e^{b^2/(8a)}
 K_{1/4}\!\left(\frac{b^2}{8a}\right).
\]
Taking \(a=\tau/4\), \(b=1/2\), and \(t=i\tau\), with \(\tau>0\),
proves \eqref{eq:quartic-Bessel-uniformisation} on the positive imaginary
axis.  Analytic continuation, followed by the boundary value on the real
axis, gives the stated formula.

Let
\[
 I_k(t)
 =
 \frac1{\sqrt{2\pi}}
 \int_\R x^k
 e^{-x^2/2+itx^4/4}\dd x.
\]
Gaussian integration by parts gives
\[
 I_2=\Phi_4+itI_4=\Phi_4+4t\Phi_4',
 \qquad
 I_4=3I_2+itI_6.
\]
Differentiating the first identity and eliminating \(I_4,I_6\) yields
\eqref{eq:quartic-ODE}.  Substitution of an exponential ansatz
\(e^{c/t}\) into the leading terms of \eqref{eq:quartic-ODE} gives
\(c(4c-i)=0\), hence the two stated exponential parts.

The continuation identities \cite[Eq.~(10.34.2), with \(m=\pm1\)]{DLMF}
\[
 K_\nu(e^{i\pi}z)
 =
 e^{-i\pi\nu}K_\nu(z)-\pi iI_\nu(z),
 \qquad
 K_\nu(e^{-i\pi}z)
 =
 e^{i\pi\nu}K_\nu(z)+\pi iI_\nu(z)
\]
give
\[
 \Phi_{\exp}^{+}-\Phi_{\exp}^{-}
 =
 2i\cos(\pi\nu)\,A K_\nu(z)
 =
 i\sqrt2\,\Phi_4,
\]
which proves \eqref{eq:quartic-Bessel-Stokes-jump}.  The Stokes and equal-magnitude rays
follow from the exponential difference \(i/(4t)\).

\end{proof}

\begin{proposition}[Parity and the real contour]
\label{prop:quartic-parity-rank}
The quartic twisted cohomology has basis \([1],[x],[x^2]\).
The monodromy orbit of the real contour spans the dual of the even
subspace and annihilates the odd subspace.  If \(\Omega\) is a nonzero
cycle dual to the odd subspace, then
\[
 H_\Omega(t)=\frac1{\sqrt{2\pi}}
 \int_\Omega x e^{-x^2/2+itx^4/4}\dd x
\]
satisfies
\begin{equation}
 H_\Omega'=-\frac{i+2t}{4t^2}H_\Omega,
 \qquad
 H_\Omega(t)=C_\Omega t^{-1/2}e^{i/(4t)},
 \qquad C_\Omega\neq0.
 \label{eq:quartic-odd-period}
\end{equation}
\end{proposition}

\begin{proof}
The operator \(\partial_x-x+itx^3\) reduces every polynomial amplitude
to a linear combination of \(1,x,x^2\). Degree comparison shows that
their classes modulo the image of this operator are linearly independent.  The connection preserves the even and odd subspaces.
The real contour annihilates odd amplitudes.

To compute a nonzero odd period, take \(t=i\tau\), \(\tau>0\), and let
\(\gamma\) run from \(i\infty\) to \(0\), then from \(0\) to \(+\infty\).
Write \(\sigma(x)=-x\) and set \(\Omega=\gamma+\sigma_*\gamma\).
Both ends of each contour are rapid-decay ends.  Under \(y=x^2\),
the contour \(\gamma\) maps to the real line with its positive orientation.
Since \(x\,\mathrm dx=\mathrm dy/2\), one obtains
\[
 \begin{split}
 H_\Omega(i\tau)
 &=\frac1{\sqrt{2\pi}}\int_{\R}e^{-y/2-\tau y^2/4}\,\mathrm dy\\
 &=\sqrt{\frac2\tau}\,e^{1/(4\tau)}\neq0.
 \end{split}
\]
The cycle \(\Omega\) annihilates even amplitudes: the pullback of
\(a(x)\,\mathrm dx\) by \(\sigma\) is \(-a(x)\,\mathrm dx\) when \(a\)
is even.  It is therefore dual to the odd subspace.

On a locally continued cycle, put
\(J_k=(2\pi)^{-1/2}\int_\Omega x^k e^{-x^2/2+itx^4/4}\,\mathrm dx\).
Integration of the derivatives of \(e^{-x^2/2+itx^4/4}\) and
\(x^2e^{-x^2/2+itx^4/4}\) gives
\[
 J_1=itJ_3,
 \qquad
 2J_1-J_3+itJ_5=0.
\]
Thus \(H_\Omega'=iJ_5/4=-(i+2t)H_\Omega/(4t^2)\), proving
\eqref{eq:quartic-odd-period}.  The value at \(i\tau\) shows
\(C_\Omega\neq0\).

For a full turn about \(z=0\), \cite[Eq.~(10.34.2), with \(m=2\)]{DLMF}
gives
\[
 K_{1/4}(ze^{2\pi i})
 =-iK_{1/4}(z)-\pi i\sqrt2\,I_{1/4}(z).
\]
The functions \(I_{1/4}\) and \(K_{1/4}\) are linearly independent, since
\cite[Eq.~(10.28.2)]{DLMF} gives
\[
 W_z(I_{1/4},K_{1/4})=-\frac1z.
\]
Continuation of \(A(t)\) around the corresponding loop multiplies it by
\(-1\).  Hence the monodromy orbit of \(A K_{1/4}\) spans the
two-dimensional even solution space.
\end{proof}

\begin{remark}[Analytic and integral bases]
\label{rem:quartic-analytic-Betti}
The multiplier \(i\sqrt2\) in
\eqref{eq:quartic-Bessel-Stokes-jump} is expressed in the Bessel-normalised
analytic frame.  The underlying Picard--Lefschetz transformations in an
integral thimble frame have integral coefficients.  This is the quartic
counterpart of the Airy normalisation in
\cref{rem:h3-formal-scales}.
\end{remark}


\section{Ranks, residue pairings, and rescaling}
\label{sec:one-variable-refinements}

\begin{example}[Monomial phases]
\label{ex:monomial-jacobi}
Let
\[
 P(x)=\frac{x^m}{m}.
\]
Then
\[
 \mathcal J_m(P)=\C[x]/(x^{m-1})
\]
with basis \(1,x,\ldots,x^{m-2}\).  The critical equation for
\(F_{P,t}\) is
\[
 -x+itx^{m-1}=0.
\]
Its solutions are \(0\) and the \(m-2\) roots of
\[
 x^{m-2}=(it)^{-1}.
\]
Thus the rapid-decay homology has rank \(m-1\), in agreement with
\eqref{eq:exact-jacobi-rank}.  At a nonzero critical point,
\[
 F_{P,t}(x)
 =
 -\frac{m-2}{2m}x^2.
\]
If \(m\) is even, the real cycle annihilates every odd basis class:
\[
 \int_\R x^{2k+1}
 \exp\!\left(-\frac{x^2}{2}+it\frac{x^m}{m}\right)\dd x=0.
\]
The full period matrix is invertible, but these odd classes have zero
period on the real contour.
\end{example}

\begin{proposition}[Rank of the monomial moment jets]
\label{prop:monomial-physical-rank}
Let \(Z\sim\mathcal N(0,1)\), \(a\in\R\setminus\{0\}\), and
\[
 \Psi_u(t)=\E\!\left[Z^u e^{itaZ^m}\right],
 \qquad 0\leq u\leq m-2.
\]
Let \(J\) be the matrix
\[
 J=\bigl(\Psi_u^{(k)}(0)\bigr)_{k\geq0,\;0\leq u\leq m-2}.
\]
Then \(\operatorname{rank}_{\C}J=r_m^{\mathrm{obs}}\), where
\begin{equation}
 r_m^{\mathrm{obs}}
 =
 \begin{cases}
  m-1,&m\ \text{odd},\\[1mm]
  m/2,&m\ \text{even}.
 \end{cases}
 \label{eq:monomial-observable-rank}
\end{equation}
Identify \(\C^{m-1}\) with the polynomials of degree at most \(m-2\)
in the monomial basis. For even \(m\), \(\ker J\) is spanned by the odd
monomials. For odd \(m\), \(\ker J=0\).
\end{proposition}

\begin{proof}
One has
\[
 \Psi_u^{(k)}(0)
 =
 (ia)^k\E Z^{mk+u}.
\]
Suppose first that \(m=2s\).  Every column indexed by an odd \(u\) vanishes.  On the even
columns \(u=2v\), \(0\leq v<s\), take \(0\leq k<s\).  Up to nonzero row
factors, the resulting matrix is
\[
 H_{kv}=\E\!\left[(Z^2)^{sk+v}\right].
\]
Let \(\nu\) be the law of \(Z^2\) on \((0,\infty)\).  Andr\'eief's
identity \cite{Andreief} gives
\[
 \det H
 =
 \frac1{s!}
 \int
 \det(y_j^{sk})_{k,j}
 \det(y_j^v)_{v,j}
 \prod_{j=1}^s\mathrm d\nu(y_j).
\]
On the chamber \(0<y_1<\cdots<y_s\), the two determinants are
Vandermonde determinants in \(y_j^s\) and \(y_j\), so they have the same
strict sign.  The integral is positive.  Thus the even sector has rank
\(s\), proving the second line of \eqref{eq:monomial-observable-rank}.

Now let \(m=2s+1\).  Use the first \(2s\) rows and reorder rows and columns
by parity.  Mixed-parity blocks vanish.  The even--even block is, up to
nonzero row factors,
\[
 \left(
  \E[(Z^2)^{(2s+1)k+v}]
 \right)_{0\leq k,v<s},
\]
and the odd--odd block is
\[
 \left(
  \E[(Z^2)^{(2s+1)k+s+1+v}]
 \right)_{0\leq k,v<s}.
\]
Apply the same Andr\'eief argument to the first block with measure
\(\mathrm d\nu(y)\), and to the second with measure
\(y^{s+1}\mathrm d\nu(y)\). Both determinants are positive.  Hence the matrix has rank \(2s=m-1\).
\end{proof}

\begin{corollary}[Diagonal multiparameter rank]
\label{cor:diagonal-physical-rank}
For
\[
 P_{\boldsymbol a}(x)
 =
 \sum_{j=1}^da_jx_j^{m_j}
\]
with independent parameters \(t_j\) in the exponential, the Jacobi,
thimble, Wick, and real-contour derivative matrices are tensor products of their
one-dimensional counterparts.  The real-cycle observable rank is
\[
 \prod_{j=1}^d r_{m_j}^{\mathrm{obs}},
\]
with \(r_m^{\mathrm{obs}}\) given by
\eqref{eq:monomial-observable-rank}.
\end{corollary}

\begin{proof}
The phase is a sum of one-variable phases. The twisted de Rham complex
and real contour are external products, and the polynomial amplitude
basis is a product basis. The mixed jet matrix is therefore a Kronecker
product, whose rank is the product of the ranks.
\end{proof}

\begin{proposition}[The residue determinant in one variable]
\label{prop:residue-determinant}
Let
\[
 g(x)=a_\mu x^\mu+a_{\mu-1}x^{\mu-1}+\cdots+a_0,
 \qquad a_\mu\neq0,
\]
and define the global residue functional on \(\C[x]/(g)\) by
\[
 \operatorname{Res}_g(f)
 =
 -\operatorname{Res}_{x=\infty}\frac{f(x)\,\mathrm d x}{g(x)}.
\]
In the basis \(1,x,\ldots,x^{\mu-1}\), the residue Gram matrix
\[
 R_g=
 \bigl(\operatorname{Res}_g(x^{i+j})\bigr)_{0\leq i,j<\mu}
\]
satisfies
\begin{equation}
 \det R_g
 =
 (-1)^{\mu(\mu-1)/2}a_\mu^{-\mu}.
 \label{eq:residue-Gram-determinant}
\end{equation}
In particular, the pairing is nondegenerate regardless of root multiplicities.

For \(P=H_3\), take
\[
 g_t(x)=x-itH_3'(x)
 =-3itx^2+x+3it.
\]
In the basis \(1,x\),
\begin{equation}
 R_{g_t}
 =
 \begin{pmatrix}
  0&i/(3t)\\
  i/(3t)&1/(9t^2)
 \end{pmatrix},
 \qquad
 \det R_{g_t}=\frac1{9t^2}.
 \label{eq:H3-residue-Gram}
\end{equation}
The matrix remains invertible at \(1-36t^2=0\).
\end{proposition}

\begin{proof}
Expansion at infinity gives
\[
 \operatorname{Res}_g(x^k)=0
 \quad(0\leq k\leq\mu-2),
 \qquad
 \operatorname{Res}_g(x^{\mu-1})=a_\mu^{-1}.
\]
Reverse the order of the columns of \(R_g\).  The resulting matrix is
triangular, with every diagonal entry equal to \(a_\mu^{-1}\).  Column
reversal has sign \((-1)^{\mu(\mu-1)/2}\), which proves
\eqref{eq:residue-Gram-determinant}.  No assumption on the discriminant of
\(g\) was used.

For \(g_t=-3itx^2+x+3it\), the first two residue identities give the
first row of \eqref{eq:H3-residue-Gram}.  Reduction modulo \(g_t\) gives
\[
 x^2=\frac{x+3it}{3it},
\]
and hence
\[
 \operatorname{Res}_{g_t}(x^2)=\frac1{9t^2}.
\]
The determinant follows directly.
\end{proof}

Balancing \(-q(x)/2\) and \(tP_m(x)\) under \(x=s^{-1}u\) gives
\(t=s^{m-2}\).  Let \(P=P_m+\cdots+P_0\) satisfy
\eqref{eq:principal-isolated}, and put
\(G(u)=-q(u)/2+iP_m(u)\).  Then
\begin{equation}
 F_{P,s^{m-2}}(s^{-1}u)
 =\frac{G(u)}{s^2}+\frac{iP_{m-1}(u)}s
  +\sum_{k=0}^{m-2}i s^{m-2-k}P_k(u).
 \label{eq:ramified-phase}
\end{equation}
The differential \(\mathfrak D_s=s^2(\mathrm d_u+
\mathrm d_uF_{P,s^{m-2}}(s^{-1}u)\wedge)\) is
\begin{equation}
 \mathfrak D_s=\mathrm dG\wedge+is\,\mathrm dP_{m-1}\wedge
 +s^2(\mathrm d_u+i\,\mathrm dP_{m-2}\wedge)
 +\sum_{k=0}^{m-3}i s^{m-k}\mathrm dP_k\wedge.
 \label{eq:Rees-differential}
\end{equation}

\begin{corollary}[The one-copy formal lattice]
\label{thm:ramified-Brieskorn-lattice}
The \(s\)-adically completed complex \((\Omega^\bullet[[s]],\mathfrak D_s)\)
deformation retracts onto
\begin{equation}
 \mathcal J(G)[[s]][-d],
 \qquad
 \mathcal J(G)=\C[u]/(\partial_1G,\ldots,\partial_dG).
 \label{eq:Rees-Jacobian}
\end{equation}
Its cohomology vanishes outside degree \(d\), and its degree-\(d\)
cohomology is a free \(\C[[s]]\)-module of rank \((m-1)^d\).
\end{corollary}

\begin{proof}
Apply \cref{lem:radial-Jacobian,thm:toric-corner-Rees} with
\(r=1\), \(S=\{1\}\),
\(\Sigma=(1)\), \(\lambda_1=1\), and \(\rho=s\).
\end{proof}

If \(G\) is Morse with distinct critical values, write its critical
points as \(c_1,\ldots,c_\mu\).  The implicit function theorem gives
critical branches \(c_\alpha(s)=c_\alpha+O(s)\), and
\(\mathrm dG(c_\alpha)=0\) gives
\[
 F_{P,s^{m-2}}(s^{-1}c_\alpha(s))
 =s^{-2}G(c_\alpha)+is^{-1}P_{m-1}(c_\alpha)+O(1).
\]
Thus \cref{prop:radial-formal-blocks} gives the exponential factors
\begin{equation}
 \exp\!\left(
 \frac{G(c_\alpha)}{s^2}
 +\frac{iP_{m-1}(c_\alpha)}s+O(1)
 \right).
 \label{eq:ramified-exponential-levels}
\end{equation}
The principal directions of switching walls and equal-magnitude rays are,
respectively,
\[
 \Im\frac{G(c_\alpha)-G(c_\beta)}{s^2}=0,
 \qquad
 \Re\frac{G(c_\alpha)-G(c_\beta)}{s^2}=0.
\]
Contour coordinates require a separate calculation; see
\cref{thm:h3-airy-stokes} for the cubic case.

\section{Cycle coordinates in one variable}
\label{sec:one-dimensional-stokes-trees}

\begin{theorem}[Thimble coordinates from boundary sectors]
\label{thm:one-dimensional-stokes-tree}
Let \(F\) be a tame Morse polynomial of degree \(m\) in one variable.
Fix an admissible direction.  Let
\[
 E_0,\ldots,E_{m-1}
\]
be the \(m\) decay sectors at infinity.  There is a canonical boundary
isomorphism
\begin{equation}
 \partial:
 H_1^{\mathrm{rd}}(F;\mathbb Z)
 \xrightarrow{\ \simeq\ }
 \widetilde H_0(\{E_0,\ldots,E_{m-1}\};\mathbb Z).
 \label{eq:one-dimensional-boundary-isomorphism}
\end{equation}
Let \(\Gamma_1,\ldots,\Gamma_{m-1}\) be a distinguished integral thimble
basis.  Define a graph \(T\) with vertices \(E_0,\ldots,E_{m-1}\) and one
edge for each thimble, directed from its initial sector to its terminal
sector.  Then \(T\) is a spanning tree.

Let \(\gamma\) be an admissible contour directed from \(E_p\) to \(E_q\).
Then
\begin{equation}
 [\gamma]
 =
 \sum_{e\in[p,q]_T}
 \epsilon_e(p,q)[\Gamma_e],
 \qquad
 \epsilon_e(p,q)\in\{-1,1\},
 \label{eq:Stokes-tree-path}
\end{equation}
where \([p,q]_T\) is the unique path from \(p\) to \(q\), and
\(\epsilon_e(p,q)\) is positive precisely when the orientation of \(e\)
agrees with that path.  Thus the thimble coordinates of \([\gamma]\)
belong to \(\{0,\pm1\}\).  If the contour lies on a Stokes
direction, the two lateral trees give the two lateral rows, related by
\eqref{eq:Betti-wall-crossing}.
\end{theorem}

\begin{proof}
The rapid-decay region outside a sufficiently large disc has \(m\)
contractible components, one in each \(E_j\).  Let \(D\) be such a closed
disc, and let \(E\) be the union of one arc of its boundary in each sector.  Truncating the ends identifies rapid-decay homology
with \(H_1(D,E;\mathbb Z)\).  Since \(D\) is contractible, the long
exact sequence of the pair \cite[Theorem~2.16]{Hatcher} gives
\[
 H_1^{\mathrm{rd}}(F;\mathbb Z)
 \simeq H_1(D,E;\mathbb Z)
 \simeq \widetilde H_0(E;\mathbb Z)
 \simeq \widetilde H_0(\{E_0,\ldots,E_{m-1}\};\mathbb Z).
\]
The boundary
of an oriented thimble with initial sector \(E_a\) and terminal sector
\(E_b\) is \(e_b-e_a\).  Since the \(m-1\) thimbles form a basis, their
reduced incidence matrix is invertible over \(\mathbb Z\).  A graph on
\(m\) vertices with \(m-1\) edges has invertible reduced incidence matrix
if and only if it is a tree.

The boundary of \(\gamma\) is \(e_q-e_p\).  In a tree, this vector is the
signed sum of the incidence vectors along the unique path from \(p\) to
\(q\).  Injectivity of the boundary map gives
\eqref{eq:Stokes-tree-path}.
\end{proof}

\begin{corollary}[Coordinates of the real contour]
\label{cor:monomial-Gaussian-Betti-row}
Let
\[
 F_{m,t}(x)=-\frac{x^2}{2}+\frac{it}{m}x^m,
 \qquad m\geq3,\quad t>0,
\]
and orient the real contour from \(-\infty\) to \(+\infty\).  Label
the decay sectors counterclockwise by the centres
\[
 \theta_j=\frac{\pi/2+2\pi j}{m},
 \qquad 0\leq j<m,
\]
and put \(\ell=\lfloor m/2\rfloor\).  For each sector, let \(L_j\) be an
oriented half-contour from a common finite base point to \(E_j\).
Define the adjacent sector basis by
\[
 \Gamma_j=-L_j+L_{j+1},
 \qquad 0\leq j\leq m-2.
\]
Then
\begin{equation}
 [\R]
 =
 [L_0-L_\ell]
 =
 -\sum_{j=0}^{\ell-1}[\Gamma_j].
 \label{eq:monomial-real-line-Betti-row}
\end{equation}
The coordinate row in this basis is
\[
 C_{\R}^{\mathrm{adj}}
 =
 (-1,\ldots,-1,0,\ldots,0),
\]
with exactly \(\ell\) entries equal to \(-1\).  If a distinguished
Lefschetz basis satisfies
\(\boldsymbol\Gamma^{\mathrm{th}}
=V\boldsymbol\Gamma^{\mathrm{adj}}\) with
\(V\in\operatorname{GL}_{m-1}(\mathbb Z)\), with both bases written as
column vectors, then
\begin{equation}
 C_{\R}^{\mathrm{th}}
 =
 C_{\R}^{\mathrm{adj}}V^{-1}.
 \label{eq:monomial-thimble-row}
\end{equation}
\end{corollary}

\begin{proof}
Rotate the positive real end to angle \(\theta_0=\pi/(2m)\) and the
negative real end to angle
\(\theta_\ell=\pi+(-1)^m\pi/(2m)\).
On each connecting arc \(x=Re^{i\theta}\),
\(\sin(m\theta)\geq0\) and \(\cos(2\theta)\geq\cos(\pi/m)>0\).
Consequently,
\[
 \Re F_{m,t}(Re^{i\theta})
 =-\frac{R^2}{2}\cos(2\theta)-\frac{tR^m}{m}\sin(m\theta)
 \leq-\frac{R^2}{2}\cos(\pi/m).
\]
The arc integrals tend to zero as \(R\to\infty\).
The real contour can therefore be deformed to a contour from \(E_\ell\)
to \(E_0\), whose boundary is \(e_0-e_\ell\).  Hence
\[
 [\R]=[L_0-L_\ell].
\]
Since
\[
 \sum_{j=0}^{\ell-1}\Gamma_j
 =
 L_\ell-L_0,
\]
\eqref{eq:monomial-real-line-Betti-row} follows.  The change-of-basis
identity gives \eqref{eq:monomial-thimble-row}.
\end{proof}

\begin{proposition}[The monodromy orbit of a monomial contour]
\label{prop:monomial-Coxeter-orbit}
Use the adjacent basis of \cref{cor:monomial-Gaussian-Betti-row}.
Let \(M_m\) be the automorphism induced by the cyclic permutation
\(E_j\mapsto E_{j-1}\), with indices modulo \(m\).  Its action is
\[
 M_m\Gamma_0=-\sum_{j=0}^{m-2}\Gamma_j,
 \qquad
 M_m\Gamma_j=\Gamma_{j-1}
 \quad(1\leq j\leq m-2).
\]
Then
\[
 \chi_{M_m}(\lambda)=\frac{\lambda^m-1}{\lambda-1}.
\]
For the class
\[
 r_m=[\R]=-\sum_{j=0}^{\ell-1}\Gamma_j,
 \qquad \ell=\lfloor m/2\rfloor,
\]
one has
\begin{equation}
 \dim_\C\Span\{M_m^kr_m:k\in\mathbb Z\}
 =
 \begin{cases}
  m-1,&m\ \mathrm{odd},\\[1mm]
  m/2,&m\ \mathrm{even}.
 \end{cases}
 \label{eq:monomial-Betti-orbit-rank}
\end{equation}
These dimensions agree with \eqref{eq:monomial-observable-rank}.
For even \(m\), the orbit span has codimension in the full homology
space equal to the dimension of the odd subspace of
\(\C[x]/(x^{m-1})\).
\end{proposition}

\begin{proof}
The displayed action cyclically rotates the \(m\) sector ends subject to
their single sum-zero relation.  Hence its eigenvalues are the nontrivial
\(m\)-th roots of unity, proving the characteristic polynomial.  Put
\(\zeta=e^{2\pi i/m}\).  In the Fourier eigenspace indexed by
\(1\leq k<m\), the component of \(r_m\), up to a nonzero factor,
is
\[
 -\sum_{j=0}^{\ell-1}\zeta^{-kj}
 =
 -\frac{1-\zeta^{-k\ell}}{1-\zeta^{-k}}.
\]
If \(m\) is odd, then \(\gcd(\ell,m)=1\), so every component is nonzero.
If \(m=2\ell\), then
\(\zeta^{-k\ell}=(-1)^k\), and the component is nonzero exactly for odd
\(k\).  There are \(\ell=m/2\) such indices.  This proves the rank
formula.  The comparison with \eqref{eq:monomial-observable-rank}
follows from the dimensions of the even and odd summands of
\(\C[x]/(x^{m-1})\).
\end{proof}

\begin{corollary}[Split phases]
\label{cor:split-Stokes-tree}
For a split phase
\[
 F(x_1,\ldots,x_d)=F_1(x_1)+\cdots+F_d(x_d)
\]
and a product contour, the Betti matrix is the Kronecker product
of the one-dimensional signed path matrices.  In particular, its entries
are in \(\{0,\pm1\}\), and its rank is the product of the ranks of the
corresponding one-dimensional matrices.
\end{corollary}

\begin{proof}
For split phases, the thimbles are products of the one-dimensional
thimbles.  The one-dimensional rapid-decay homology groups are free,
so the K\"unneth isomorphism \cite[Theorem~3B.5]{Hatcher} identifies
product-thimble classes with tensor products.
Expanding each contour in its thimble basis therefore gives the
Kronecker product of the coefficient matrices.
\end{proof}

\section{Cyclicity of a selected contour}
\label{sec:selected-contour-cyclicity}

Let \((\mathcal H_P,\nabla)\) be a meromorphic connection on a connected
ordinary domain \(U\), and let \(\mathcal W_P\subset\mathcal H_P\) be a
monodromy-stable subspace of rank \(R\).  Fix a basis
\(\omega_1,\ldots,\omega_R\) of a fibre of \(\mathcal W_P\) and a
rapid-decay cycle \(\Gamma_P\).  Write \(\rho_P\) for
dual monodromy.  The cycle is cyclic on \(\mathcal W_P\) precisely when
there are loops \(\ell_1,\ldots,\ell_R\) such that
\begin{equation}
 \Delta_{\mathrm{per}}(P)
 :=\det\Pi_P\neq0,
 \qquad
 \Pi_P=\left(
 \ip{\rho_P(\ell_i)\Gamma_P}{\omega_j}
 \right)_{1\leq i,j\leq R}.
 \label{eq:period-cyclicity}
\end{equation}
Indeed, the rows span \(\mathcal W_P^\vee\) if and only if the determinant
is nonzero.  Equivalently,
\(\Pi_P^*\Pi_P>0\), since
\(\det(\Pi_P^*\Pi_P)=|\Delta_{\mathrm{per}}(P)|^2\).
For fixed loops continued in a holomorphic family on a tame stratum,
\(\Delta_{\mathrm{per}}\) is holomorphic, so its nonvanishing locus is open.

For \(P(x)=x^4\), the substitution \(T=4t\) in
\cref{prop:quartic-parity-rank} gives the quotient
\[
 \C(t)[x]/(\partial_x-x+4itx^3)\C(t)[x]
 =\Span_{\C(t)}\{[1],[x],[x^2]\}.
\]
Its connection \(\nabla_t[a]=[\partial_ta+ix^4a]\) preserves parity.
The period of the nonzero class \([x]\) over the real contour and all
parameter derivatives of that period vanish:
\[
 \frac1{\sqrt{2\pi}}\int_\R
 x^{4k+1}e^{-x^2/2+itx^4}\dd x=0,
 \qquad k\geq0.
\]
Its monodromy orbit therefore has rank two.

\section*{Declarations}

\noindent\textbf{Funding.}
This research did not receive any specific grant from funding agencies in
the public, commercial, or not-for-profit sectors.

\medskip
\noindent\textbf{Competing interests.}
The author declares no competing interests.

\medskip
\noindent\textbf{Data availability.}
No data were used for the research described in this article.

\section*{Declaration of generative AI and AI-assisted technologies in the
manuscript preparation process}
Large language models were used initially to help relate calculations developed by the author to relevant mathematical frameworks and references. ChatGPT
also identified a problem in an earlier version of one theorem concerning the separation of
critical values, which led to the distinction between central nonresonance and the stronger
pairwise-distinct-critical-values assumption used in the final version. ChatGPT was subsequently used
to assist with the organisation and critical review of the manuscript. The author takes full
responsibility for all mathematical content. The core ideas, theoretical framework, and key
derivations are entirely due to the author.


\bigskip
\footnotesize
\begin{thebibliography}{99}

\bibitem{Andreief}
C.~Andr\'eief,
\emph{Note sur une relation entre les int\'egrales d\'efinies des produits
des fonctions},
M\'em. Soc. Sci. Phys. Nat. Bordeaux (3) \textbf{2} (1886), 1--14.

\bibitem{Brieskorn}
E.~Brieskorn,
\emph{Die Monodromie der isolierten Singularit\"aten von Hyperfl\"achen},
Manuscripta Math. \textbf{2} (1970), 103--161.
\url{https://doi.org/10.1007/BF01155695}

\bibitem{DAgnoloHienMorandoSabbah}
A.~D'Agnolo, M.~Hien, G.~Morando, and C.~Sabbah,
\emph{Topological computation of some Stokes phenomena on the affine line},
Ann. Inst. Fourier (Grenoble) \textbf{70} (2020), no.~2, 739--808.
\url{https://doi.org/10.5802/aif.3323}

\bibitem{DerksenBounds}
H.~Derksen,
\emph{Polynomial bounds for rings of invariants},
Proc. Amer. Math. Soc. \textbf{129} (2001), no.~4, 955--963.
\url{https://doi.org/10.1090/S0002-9939-00-05698-7}

\bibitem{DerksenKemper}
H.~Derksen and G.~Kemper,
\emph{Computational Invariant Theory},
second enlarged ed., Encyclopaedia of Mathematical Sciences, vol.~130,
Springer, 2015.
\url{https://doi.org/10.1007/978-3-662-48422-7}

\bibitem{DimcaSaito}
A.~Dimca and M.~Saito,
\emph{Algebraic Gauss--Manin systems and Brieskorn modules},
Amer. J. Math. \textbf{123} (2001), no.~1, 163--184.
\url{https://doi.org/10.1353/ajm.2001.0001}

\bibitem{DLMF}
F.~W.~J. Olver, D.~W. Lozier, R.~F. Boisvert, and C.~W. Clark (eds.),
\emph{NIST Handbook of Mathematical Functions},
Cambridge University Press and NIST, 2010.
\url{https://dlmf.nist.gov/}

\bibitem{DouaiSabbah}
A.~Douai and C.~Sabbah,
\emph{Gauss--Manin systems, Brieskorn lattices and Frobenius structures (I)},
Ann. Inst. Fourier (Grenoble) \textbf{53} (2003), no.~4, 1055--1116.
\url{https://doi.org/10.5802/aif.1974}

\bibitem{GugenheimLambe}
V.~K.~A.~M. Gugenheim and L.~A. Lambe,
\emph{Perturbation theory in differential homological algebra. I},
Illinois J. Math. \textbf{33} (1989), no.~4, 566--582.
\url{https://doi.org/10.1215/ijm/1255988571}

\bibitem{GugenheimLambeStasheff}
V.~K.~A.~M. Gugenheim, L.~A. Lambe, and J.~D. Stasheff,
\emph{Perturbation theory in differential homological algebra. II},
Illinois J. Math. \textbf{35} (1991), no.~3, 357--373.
\url{https://doi.org/10.1215/ijm/1255987784}

\bibitem{Hatcher}
A.~Hatcher,
\emph{Algebraic Topology},
Cambridge University Press, Cambridge, 2002.
\url{https://pi.math.cornell.edu/~hatcher/AT/AT.pdf}

\bibitem{HelmerHongHong}
M.~Helmer, D.~Hong, and H.~Hong,
\emph{Certificate for orthogonal equivalence of real polynomials by
polynomial-weighted principal component analysis},
J. Symbolic Comput. \textbf{138} (2027), article 102586;
published online 2026.
\url{https://doi.org/10.1016/j.jsc.2026.102586}

\bibitem{Hien}
M.~Hien,
\emph{Periods for flat algebraic connections},
Invent. Math. \textbf{178} (2009), no.~1, 1--22.
\url{https://doi.org/10.1007/s00222-009-0185-7}

\bibitem{HienRoucairol}
M.~Hien and C.~Roucairol,
\emph{Integral representations for solutions of exponential
Gau{\ss}--Manin systems},
Bull. Soc. Math. France \textbf{136} (2008), no.~4, 505--532.
\url{https://doi.org/10.24033/bsmf.2564}

\bibitem{Hohl}
A.~Hohl,
\emph{(D)-modules of pure Gaussian type and enhanced ind-sheaves},
Manuscripta Math. \textbf{167} (2022), 435--467.
\url{https://doi.org/10.1007/s00229-021-01281-y}

\bibitem{Isserlis}
L.~Isserlis,
\emph{On a formula for the product-moment coefficient of any order of a
normal frequency distribution in any number of variables},
Biometrika \textbf{12} (1918), no.~1/2, 134--139.
\url{https://doi.org/10.1093/biomet/12.1-2.134}

\bibitem{Janson}
S.~Janson,
\emph{Gaussian Hilbert Spaces},
Cambridge Tracts in Mathematics, vol.~129,
Cambridge University Press, Cambridge, 1997.
\url{https://doi.org/10.1017/CBO9780511526169}

\bibitem{KurdykaSpodzieja}
K.~Kurdyka and S.~Spodzieja,
\emph{Separation of real algebraic sets and the \L ojasiewicz exponent},
Proc. Amer. Math. Soc. \textbf{142} (2014), no.~9, 3089--3102.
\url{https://doi.org/10.1090/S0002-9939-2014-12061-2}

\bibitem{Milnor}
J.~Milnor,
\emph{Singular Points of Complex Hypersurfaces},
Annals of Mathematics Studies, vol.~61,
Princeton University Press, Princeton, NJ, 1968.

\bibitem{PeccatiTaqqu}
G.~Peccati and M.~S. Taqqu,
\emph{Wiener Chaos: Moments, Cumulants and Diagrams: A Survey with
Computer Implementation},
Bocconi \& Springer Series, vol.~1, Springer, 2011.

\bibitem{PflugelSplitting}
E.~Pfl\"ugel,
\emph{A Root-Free Splitting-Lemma for Systems of Linear Differential Equations},
arXiv:1911.05837, 2019.
\url{https://doi.org/10.48550/arXiv.1911.05837}

\bibitem{Pham}
F.~Pham,
\emph{Vanishing homologies and the \(n\)-variable saddlepoint method},
in \emph{Singularities, Part 2 (Arcata, Calif., 1981)},
Proc. Sympos. Pure Math., vol.~40, part~2,
American Mathematical Society, 1983, pp.~319--333.

\bibitem{ProcesiSchwarz}
C.~Procesi and G.~Schwarz,
\emph{Inequalities defining orbit spaces},
Invent. Math. \textbf{81} (1985), 539--554.
\url{https://doi.org/10.1007/BF01388587}

\bibitem{SabbahStokes}
C.~Sabbah,
\emph{Introduction to Stokes Structures},
Lecture Notes in Mathematics, vol.~2060,
Springer, Heidelberg, 2013.
\url{https://doi.org/10.1007/978-3-642-31695-1}

\bibitem{SabbahTwisted}
C.~Sabbah,
\emph{On a twisted de Rham complex},
Tohoku Math. J. (2) \textbf{51} (1999), no.~1, 125--140.
\url{https://doi.org/10.2748/tmj/1178224856}

\bibitem{SchrijverEdge}
A.~Schrijver,
\emph{Graph invariants in the edge model},
in \emph{Building Bridges: Between Mathematics and Computer Science},
M.~Gr\"otschel, G.~O.~H. Katona, and G.~S\'agi (eds.),
Springer, Berlin, Heidelberg, 2008, pp.~487--498.
\url{https://doi.org/10.1007/978-3-540-85221-6_16}

\bibitem{Schwarz}
G.~W. Schwarz,
\emph{Smooth functions invariant under the action of a compact Lie group},
Topology \textbf{14} (1975), no.~1, 63--68.
\url{https://doi.org/10.1016/0040-9383(75)90036-1}

\bibitem{StacksProject}
The Stacks Project Authors,
\emph{The Stacks Project},
Tags 00N7 and 062F.
\url{https://stacks.math.columbia.edu/}

\bibitem{Weyl}
H.~Weyl,
\emph{The Classical Groups: Their Invariants and Representations},
Princeton University Press, 1939.

\bibitem{YoungHolant}
B.~Young,
\emph{The converse of the real orthogonal Holant theorem},
in \emph{52nd International Colloquium on Automata, Languages, and
Programming (ICALP 2025)},
Leibniz Int. Proc. Inform., vol.~334,
Schloss Dagstuhl--Leibniz-Zentrum f\"ur Informatik, 2025,
pp.~136:1--136:20.
\url{https://doi.org/10.4230/LIPIcs.ICALP.2025.136}

\end{thebibliography}
\end{document}